\documentclass[reqno,12pt]{amsart}

\usepackage{fullpage}
\usepackage[english]{babel}
\usepackage[T1]{fontenc}
\usepackage{graphicx}
\usepackage{amsmath}
\usepackage{amsfonts}
\usepackage{amssymb}
\usepackage{a4wide}

\newtheorem{theo}{Theorem}
\newtheorem*{conjA}{Conjecture A}
\newtheorem*{conjB}{Conjecture B}
\newtheorem*{conjC}{Conjecture C}

\newtheorem{propo}{Proposition}
\newtheorem{lemme}{Lemma}

\theoremstyle{remark}
\newtheorem*{Remark}{Remark}

\numberwithin{equation}{section}

\author{\'E. DELAYGUE}
\title{Arithmetic properties of Ap\'ery-like numbers}
\date{}



\thanks{Research supported by the project Holonomix (PEPS CNRS INS2I 2012).}

\begin{document}
\maketitle

\begin{abstract}
We provide lower bounds for $p$-adic valuations of multisums of factorial ratios which satisfy an Ap\'ery-like recurrence relation: these include Ap\'ery, Domb, Franel numbers, the numbers of abelian squares over a finite alphabet, and constant terms of powers of certain Laurent polynomials. In particular, we prove Beukers' conjectures on the $p$-adic valuation of Ap\'ery numbers. Furthermore, we give an effective criterion for a sequence of factorial ratios to satisfy the $p$-Lucas property for almost all primes $p$.
\end{abstract}

\section{Introduction}

\subsection{Classical results of Lucas and Kummer}

It is a well-known result of Lucas \cite{Lucas} that, for all nonnegative integers $m,n$ and all primes $p$, we have
\begin{equation}\label{RealLucas}
\binom{m}{n}\equiv\prod_{i=0}^k\binom{m_i}{n_i}\mod p,
\end{equation}
where $m=m_0+m_1p+\cdots+m_kp^k$ and $n=n_0+n_1p+\cdots+n_kp^k$ are the base $p$ expansions of $m$ and $n$. 

In particular, a prime $p$ divides the binomial $\binom{m}{n}$ if, and only if there is $0\leq i\leq k$ such that $m_i<n_i$. Precisely, Kummer proved in \cite{Kummer} that, for all natural integers $m\geq n$, the $p$-adic valuation (\footnote{The $p$-adic valuation of an integer $m$ is the maximum integer $\beta$ such that $p^\beta$ divides $m$.}) of the binomial $\binom{m}{n}$ is the number of carries which occur when $n$ is added to $m-n$ in base $p$. As a consequence, we have
\begin{equation}\label{OldConj}
\binom{m}{n}\in p^\alpha\mathbb{Z},\quad\textup{where}\quad\alpha=\#\left\{0\leq i\leq k\,:\,\binom{m_i}{n_i}=0\right\}.
\end{equation}

In this article, we show that many sequences of Ap\'ery-like numbers satisfy congruences similar to \eqref{RealLucas}, that is, for all nonnegative integers $n$ and all primes $p$, we have
$$
A(n)\equiv\prod_{i=0}^k A(n_i)\mod p,
$$
where $n=n_0+n_1p+\cdots+n_kp^k$ is the base $p$ expansion of $n$. Furthermore, we prove that an analogue of \eqref{OldConj} holds for those numbers, that is
$$
A(n)\in p^\alpha\mathbb{Z},\quad\textup{where}\;\alpha=\#\big\{0\leq i\leq k\,:\,A(n_i)\equiv 0\mod p\big\},
$$
which proves Beukers' conjectures on the $p$-adic valuation of Ap\'ery numbers.

\subsection{Beukers' conjectures on Ap\'ery numbers}

For all natural integers $n$, we set 
$$
A_1(n):=\sum_{k=0}^n\binom{n}{k}^2\binom{n+k}{k}^2\quad\textup{and}\quad A_2(n):=\sum_{k=0}^n\binom{n}{k}^2\binom{n+k}{k}.
$$

Those sequences were used in $1979$ by Ap\'ery in his proofs of the irrationality of $\zeta(3)$ and $\zeta(2)$ (see \cite{Ap\'ery}). In the $1980$'s, several congruences satisfied by those sequences were demonstrated (see for example \cite{BeukersAp1}, \cite{BeukersAp2}, \cite{Chowla}, \cite{Gessel}, \cite{Mimura}). In particular, Gessel proved in \cite{Gessel} that $A_1$ satisfies the $p$-Lucas property for all prime numbers $p$, that is, for any prime $p$, all $v$ in $\{0,\dots,p-1\}$ and all natural integers $n$, we have
$$
A_1(v+np)\equiv A_1(v)A_1(n)\mod p.
$$
Thereby, if $n=n_0+n_1p+\cdots+n_Np^N$ is the base $p$ expansion of $n$, then we obtain
\begin{equation}\label{debase}
A_1(n)\equiv A_1(n_0)\cdots A_1(n_N)\mod p.
\end{equation}
In particular, $p$ divides $A_1(n)$ if and only if there exists $k$ in $\{0,\dots,N\}$ such that $p$ divides $A_1(n_k)$. Beukers stated in \cite{Beukers} two conjectures, when $p=5$ or $11$, which generalize this property (\footnote{If $p$ is $2$, $3$ or $7$, then for all $v$ in $\{0,\dots,p-1\}$, $A_1(v)$ is coprime to $p$ so that, according to \eqref{debase}, for all natural integers $n$, $A_1(n)$ is coprime to $p$.}). Before stating these conjectures, we observe that the set of all $v$ in $\{0,\dots,4\}$ (respectively $v$ in $\{0,\dots,10\}$) satisfying $A_1(v)\equiv 0\mod 5$ (respectively $A_1(v)\equiv 0\mod 11$) is $\{1,3\}$ (respectively $\{5\}$).

\begin{conjA}[Beukers, \cite{Beukers}]\label{conj1}
Let $n$ be a natural integer whose base $5$ expansion is $n=n_0+n_15+\cdots+n_N5^N$. Let $\alpha$ be the number of $k$ in $\{0,\dots,N\}$ such that $n_k=1$ or $3$. Then $5^\alpha$ divides $A_1(n)$.
\end{conjA}

\begin{conjB}[Beukers, \cite{Beukers}]\label{conj2}
Let $n$ be a natural integer whose base $11$ expansion is $n=n_0+n_111+\cdots+n_N11^N$. Let $\alpha$ be the number of $k$ in $\{0,\dots,N\}$ such that $n_k=5$. Then $11^\alpha$ divides $A_1(n)$.
\end{conjB}

Similarly, Sequence $A_2$ satisfies the $p$-Lucas property for all primes $p$. Furthermore, Beukers and Stienstra proved in \cite{BeukersStienstra} that, if $p\equiv 3\mod 4$, then $A_2\left(\frac{p-1}{2}\right)\equiv 0\mod p$, and Beukers stated in \cite{Beukers} the following conjecture.

\begin{conjC}
Let $p$ be a prime number satisfying $p\equiv 3\mod 4$. Let $n$ be a natural integer whose base $p$ expansion is $n=n_0+n_1p+\cdots+n_Np^N$. Let $\alpha$ be the number of $k$ in $\{0,\dots,N\}$ such that $n_k=\frac{p-1}{2}$. Then $p^\alpha$ divides $A_2(n)$.
\end{conjC}

Conjectures A-C have been extended to generalized Ap\'ery numbers and any prime $p$ by Deutsch and Sagan in \cite[Conjecture 5.13]{DeutschSagan} but this conjecture is false for at least one generalization of Ap\'ery numbers. Indeed, a counterexample is given by
$$
A(n)=\sum_{k=0}^n\binom{n}{k}^2\binom{n+k}{k}^3,
$$
since $A(1)=9\equiv 0\mod 3$ but $A(4)=A(1+3)=1152501$ is not divisible by $3^2$.
\medskip

The main aim of this article is to prove Theorem \ref{theo gene}, stated in Section \ref{section main}, which demonstrates and generalizes Conjectures A-C. First, we introduce some notations which we use throughout this article.

\subsection{Notations}

Let $d$ be a positive integer. If $\mathbf{m}=(m_1,\dots,m_d)$ and $\mathbf{n}=(n_1,\dots,n_d)$ belong to $\mathbb{R}^d$ and if $\lambda\in\mathbb{R}$ and $k\in\{1,\dots,d\}$, then we write:
\begin{itemize}
\item $\mathbf{m}+\mathbf{n}=(m_1+n_1,\dots,m_d+n_d)$;
\item $\mathbf{m}\cdot\mathbf{n}=m_1n_1+\cdots+m_dn_d$;
\item $\mathbf{m}\lambda=(m_1\lambda,\dots,m_d\lambda)$;
\item $|\mathbf{m}|=m_1+\cdots+m_d$;
\item $\mathbf{m}^{(k)}=m_k$;
\item $\mathbf{m}\geq\mathbf{n}$ if, and only if, for all $i$ in $\{1,\dots,d\}$, we have $m_i\geq n_i$.
\end{itemize}
Furthermore, we set $\mathbf{0}=(0,\dots,0)\in\mathbb{N}^d$, $\mathbf{1}=(1,\dots,1)\in\mathbb{N}^d$ and we write $\mathbf{1}_k$ for the vector in $\mathbb{N}^d$, all of whose coordinates equal zero except the $k$-th which is $1$. If $p$ is a prime number and $\mathbf{n}$ is nonzero, then we say that $\mathbf{n}=\mathbf{n}_0+\mathbf{n}_1p+\cdots+\mathbf{n}_Np^N$ is the base $p$ expansion of $\mathbf{n}$ if, for all $i$ in $\{0,\dots,N\}$, we have $\mathbf{n}_i\in\{0,\dots,p-1\}^d$, and $\mathbf{n}_N\neq\mathbf{0}$.
\medskip 

For all primes $p$, we write $\mathbb{Z}_p$ for the ring of $p$-adic integers. If $A=\big(A(\mathbf{n})\big)_{\mathbf{n}\in\mathbb{N}^d}$ is a $\mathbb{Z}_p$-valued family, then we say that $A$ satisfies the $p$-Lucas property if and only if, for all vectors $\mathbf{v}$ in $\{0,\dots,p-1\}^d$ and $\mathbf{n}$ in $\mathbb{N}^d$, we have
\begin{equation}\label{FastRef}
A(\mathbf{v}+\mathbf{n}p)\equiv A(\mathbf{v})A(\mathbf{n})\mod p\mathbb{Z}_p.
\end{equation}
We write $f_A$ for the generating series of $A$ defined by $f_A(\mathbf{z}):=\sum_{\mathbf{n}\in\mathbb{N}^d} A(\mathbf{n})\mathbf{z}^{\mathbf{n}}$, where, if $\mathbf{z}=(z_1,\dots,z_d)$ is a vector of variables and $\mathbf{n}=(n_1,\dots,n_d)\in\mathbb{N}^d$, $\mathbf{z}^{\mathbf{n}}$ denotes $z_1^{n_1}\cdots z_d^{n_d}$.
\medskip

In addition, we write $\mathcal{Z}_p(A)$ for the set of all vectors $\mathbf{v}$ in $\{0,\dots,p-1\}^d$ such that $A(\mathbf{v})\in p\mathbb{Z}_p$. For every nonzero vector $\mathbf{n}$ in $\mathbb{N}^d$ whose base $p$ expansion is $\mathbf{n}=\mathbf{n}_0+\mathbf{n}_1p+\cdots+\mathbf{n}_Np^N$, we write $\alpha_p(A,\mathbf{n})$ for the number of $k$ in $\{0,\dots,N\}$ such that $\mathbf{n}_k\in\mathcal{Z}_p(A)$, and we set $\alpha_p(A,\mathbf{0})=0$. Thereby, to prove Conjectures A-C, it is enough to show that $A_i(n)\in p^{\alpha_p(A_i,n)}\mathbb{Z}$ with $i=1,p=5$ or $11$ and $i=2$, $p\equiv 3\mod 4$.
\medskip

Given tuples of vectors in $\mathbb{N}^d$, $e=(\mathbf{e}_1,\dots,\mathbf{e}_u)$ and $f=(\mathbf{f}_1,\dots,\mathbf{f}_v)$, we write $|e|=\sum_{i=1}^u\mathbf{e}_i$ and, for all vectors $\mathbf{n}$ in $\mathbb{N}^d$ and all natural integers $m$, we set
$$
\mathcal{Q}_{e,f}(\mathbf{n}):=\frac{\prod_{i=1}^u(\mathbf{e}_i\cdot\mathbf{n})!}{\prod_{i=1}^v(\mathbf{f}_i\cdot\mathbf{n})!}\quad\textup{and}\quad\mathfrak{S}_{e,f}(m):=\sum_{\mathbf{n}\in\mathbb{N}^d,|\mathbf{n}|=m}\mathcal{Q}_{e,f}(\mathbf{n}).
$$
\medskip

Let $\mathcal{S}:=\{1\leq i\leq u\,:\,\mathbf{e}_i\geq\mathbf{1}\}$. For every positive integer $r$, we say that $e$ is \textit{$r$-admissible} if
$$
\#\mathcal{S}+\min_{1\leq k\leq d}\#\{1\leq i\leq u\,:\,i\notin\mathcal{S},\,\mathbf{e}_i\geq d\mathbf{1}_k\}\geq r.
$$
\medskip

For all primes $p$, we write $\mathfrak{F}_p^d$ for the set of all functions $g:\mathbb{N}^d\rightarrow\mathbb{Z}_p$ such that, for all natural integers $K$, there exists a sequence $(P_{K,k})_{k\geq 0}$ of polynomial functions with coefficients in $\mathbb{Z}_p$ which converges pointwise to $g$ on $\{0,\dots,K\}^d$. For all tuples $e$ and $f$ of vectors in $\mathbb{N}^d$, all $g\in\mathfrak{F}_p^d$ and all natural integers $m$, we set
$$
\mathfrak{S}_{e,f}^g(m):=\sum_{\mathbf{n}\in\mathbb{N}^d,|\mathbf{n}|=m}\mathcal{Q}_{e,f}(\mathbf{n})g(\mathbf{n}).
$$
\medskip

Finally, we set $\theta:=z\frac{d}{dz}$ and we say that a differential operator $\mathcal{L}$ in $\mathbb{Z}_p[z,\theta]$ is of \textit{type I} if there is a natural integer $q$ such that:
\begin{itemize}
\item $\mathcal{L}=P_0(\theta)+zP_1(\theta)+\cdots+z^qP_q(\theta)$ with $P_k(X)\in\mathbb{Z}_p[X]$ for $0\leq k\leq q$;
\item  $P_0(\mathbb{Z}_p^\times)\subset\mathbb{Z}_p^\times$;
\item for all $k$ in $\{2,\dots,q\}$, we have $P_k(X)\in\prod_{i=1}^{k-1}(X+i)^2\mathbb{Z}_p[X]$.
\end{itemize}
We say that a differential operator $\mathcal{L}$ in $\mathbb{Z}_p[z,\theta]$ is of \textit{type II} if
\begin{itemize}
\item $\mathcal{L}=P_0(\theta)+zP_1(\theta)+z^2P_2(\theta)$ with $P_k(X)\in\mathbb{Z}_p[X]$ for $0\leq k\leq 2$; 
\item $P_0(\mathbb{Z}_p^\times)\subset\mathbb{Z}_p^\times$;
\item $P_2(X)\in(X+1)\mathbb{Z}_p[X]$.
\end{itemize}

\subsection{Main results}\label{section main}

The main result of this article is the following.

\begin{theo}\label{theo gene}
Let $e$ and $f=(\mathbf{1}_{k_1},\dots,\mathbf{1}_{k_v})$ be two disjoint tuples of vectors in $\mathbb{N}^d$ such that $|e|=|f|$, for all $i$ in $\{1,\dots,v\}$, $k_i$ is in $\{1,\dots,d\}$, and $e$ is $2$-admissible. Let $p$ be a fixed prime. Assume that $f_{\mathfrak{S}_{e,f}}$ is annihilated by a differential operator $\mathcal{L}\in\mathbb{Z}_p[z,\theta]$ such that at least one of the following conditions holds:
\begin{itemize}
\item $\mathcal{L}$ is of type I.
\item $\mathcal{L}$ is of type II and $p-1\in\mathcal{Z}_p(\mathfrak{S}_{e,f})$. 
\end{itemize}
Then, for all natural integers $n$ and all functions $g$ in $\mathfrak{F}_p^d$, we have
$$
\mathfrak{S}_{e,f}(n)\in p^{\alpha_p(\mathfrak{S}_{e,f},n)}\mathbb{Z}\quad\textup{and}\quad\mathfrak{S}_{e,f}^g\in p^{\alpha_p(\mathfrak{S}_{e,f},n)-1}\mathbb{Z}_p.
$$
\end{theo}

In Section \ref{Application}, we show that Theorem \ref{theo gene} applies to many classical sequences. In particular, Theorem \ref{theo gene} implies Conjectures A-C. Indeed, we have $A_1=\mathfrak{S}_{e_1,f_1}$ and $A_2=\mathfrak{S}_{e_2,f_2}$ with $d=2$, 
$$
e_1=\big((2,1),(2,1)\big)\quad\textup{and}\quad f_1=\big((1,0),(1,0),(1,0),(1,0),(0,1),(0,1)\big),
$$
and 
$$
e_2=\big((2,1),(1,1)\big)\quad\textup{and}\quad f_2=\big((1,0),(1,0),(1,0),(0,1),(0,1)\big). 
$$
Furthermore, it is well known that $f_{A_1}$, respectively $f_{A_2}$, is annihilated by the differential operator $\mathcal{L}_1$, respectively $\mathcal{L}_2$, defined by
$$
\mathcal{L}_1=\theta^3-z(34\theta^3+51\theta^2+27\theta+5)+z^2(\theta+1)^3
$$
and
$$
\mathcal{L}_2=\theta^2-z(11\theta^2+11\theta+3)-z^2(\theta+1)^2.
$$
Since $\mathcal{L}_1$ and $\mathcal{L}_2$ are of type I for all primes $p$, the conditions of Theorem \ref{theo gene} are satisfied by $A_1$ and $A_2$, and Conjectures A-C hold. In addition, for all primes $p$ and all natural integers $n$ and $\alpha$, we obtain that
$$
\sum_{k=0}^nk^\alpha\binom{n}{k}^2\binom{n+k}{k}^2\in p^{\alpha_p(A_1,n)-1}\mathbb{Z}\quad\textup{and}\quad\sum_{k=0}^n k^\alpha\binom{n}{k}^2\binom{n+k}{k}\in p^{\alpha_p(A_2,n)-1}\mathbb{Z}.
$$
\medskip

We provide a similar result which applies to the constant terms of powers of certain Laurent polynomials. Consider a Laurent polynomial 
$$
\Lambda(\mathbf{x})=\sum_{i=1}^k\alpha_i\mathbf{x}^{\mathbf{a}_i}\in\mathbb{Z}_p[x_1^\pm,\dots,x_d^\pm],
$$
where $\mathbf{a}_i\in\mathbb{Z}^d$ and $\alpha_i\neq 0$ for $i$ in $\{1,\dots,k\}$. Recall that the Newton polyhedron of $\Lambda$ is the convex hull of $\{\mathbf{a}_1,\dots,\mathbf{a}_k\}$ in $\mathbb{R}^d$. Hence we have the following result.

\begin{theo}\label{Laurent Ap}
Let $p$ be a fixed prime. Let $\Lambda(\mathbf{x})\in\mathbb{Z}_p[x_1^{\pm},\dots,x_d^{\pm}]$ be a Laurent polynomial, and consider the sequence of the constant terms of powers of $\Lambda$ defined, for all natural integers $n$, by
$$
A(n):=\big[\Lambda(\mathbf{x})^n\big]_{\mathbf{0}}.
$$
Assume that the Newton polyhedron of $\Lambda$ contains the origin as its only interior integral point, and that
$f_A$ is annihilated by a differential operator $\mathcal{L}$ in $\mathbb{Z}_p[z,\theta]$ such that at least one of the following conditions holds:
\begin{itemize}
\item $\mathcal{L}$ is of type I.
\item $\mathcal{L}$ is of type II and $p-1\in\mathcal{Z}_p(A)$. 
\end{itemize}
Then, for all natural integers $n$, we have 
$$
A(n)\in p^{\alpha_p(A,n)}\mathbb{Z}_p.
$$
\end{theo}

For example, Theorem \ref{Laurent Ap} applies to Apéry numbers $A_1$ thanks to the following formula of Lairez \cite{Lairez}:
$$
A_1(n)=\left[\left(\frac{(1+z)(yz+z+1)(1+x)(xy+x+y)}{xyz}\right)^n\right]_{(0,0,0)}.
$$
\medskip

By a result of Samol and van Straten \cite{Samol}, if $\Lambda(\mathbf{x})\in\mathbb{Z}_p[x_1^{\pm},\dots,x_d^{\pm}]$ contains the origin as its only interior integral point, then $\big([\Lambda(\mathbf{x})^n]_{\mathbf{0}}\big)_{n\geq 0}$ satisfies the $p$-Lucas property, which is essential for the proof of Theorem \ref{Laurent Ap}. Likewise, the proof of Theorem \ref{theo gene} rests on the fact that $\mathfrak{S}_{e,f}$ satisfies the $p$-Lucas property when $|e|=|f|$, $e$ is $2$-admissible and $f=(\mathbf{1}_{k_1},\dots,\mathbf{1}_{k_v})$. Since those results deal with multisums of factorial ratios, it seems natural to study similar arithmetic properties for simpler numbers such as families of factorial ratios. To that purpose, we prove Theorem \ref{FactoCrit} below which gives an effective criterion for $\mathcal{Q}_{e,f}$ to satisfy the $p$-Lucas property for almost all primes $p$ (\footnote{Throughout this article, we say that an assertion $\mathcal{A}_p$ is true for almost all primes $p$ if there exists a constant $C\in\mathbb{N}$ such that $\mathcal{A}_p$ holds for all primes $p\geq C$.}). Furthermore, Theorem \ref{FactoCrit} shows that if $A:=\mathcal{Q}_{e,f}$ satisfies the $p$-Lucas property for almost all primes $p$, then, for all natural integers $n$ and all primes $p$, we have $A(n)\in p^{\alpha_p(A,n)}\mathbb{Z}$. 
\medskip

To state this result, we introduce some additional notations. For all tuples $e$ and $f$ of vectors in $\mathbb{N}^d$, we write $\Delta_{e,f}$ for Landau's function defined, for all $\mathbf{x}$ in $\mathbb{R}^d$, by
$$
\Delta_{e,f}(\mathbf{x}):=\sum_{i=1}^u\lfloor \mathbf{e}_i\cdot\mathbf{x}\rfloor-\sum_{i=1}^v\lfloor \mathbf{f}_i\cdot\mathbf{x}\rfloor\in\mathbb{Z},
$$
where $\lfloor\cdot\rfloor$ denotes the floor function. Therefore, according to Landau's criterion \cite{Landau} and a precision of the author \cite{Delaygue3}, we have the following dichotomy.
\begin{itemize}
\item If, for all $\mathbf{x}$ in $[0,1]^d$, we have $\Delta_{e,f}(\mathbf{x})\geq 0$, then $\mathcal{Q}_{e,f}$ is a family of integers;
\item if there exists $\mathbf{x}$ in $[0,1]^d$ such that $\Delta_{e,f}(\mathbf{x})\leq -1$, then there are only finitely many primes $p$ such that $\mathcal{Q}_{e,f}$ is a family of $p$-adic integers. 
\end{itemize}

In the rest of the article, we write $\mathcal{D}_{e,f}$ for the semi-algebraic set of all $\mathbf{x}$ in $[0,1)^d$ such that there exists a component $\mathbf{d}$ of $e$ or $f$ satisfying $\mathbf{d}\cdot\mathbf{x}\geq 1$. Observe that $\Delta_{e,f}$ vanishes on the nonempty set $[0,1)^d\setminus\mathcal{D}_{e,f}$.

\begin{theo}\label{FactoCrit}
Let $e$ and $f$ be disjoint tuples of vectors in $\mathbb{N}^d$ such that $\mathcal{Q}_{e,f}$ is a family of integers. Then we have the following dichotomy.
\begin{enumerate}
\item If $|e|=|f|$ and if, for all $\mathbf{x}$ in $\mathcal{D}_{e,f}$, we have $\Delta_{e,f}(\mathbf{x})\geq 1$, then for all primes $p$, $\mathcal{Q}_{e,f}$ satisfies the $p$-Lucas property;
\item if $|e|\neq|f|$ or if there exists $\mathbf{x}$ in $\mathcal{D}_{e,f}$ such that $\Delta_{e,f}(x)=0$, then there are only finitely many primes $p$ such that $\mathcal{Q}_{e,f}$ satisfies the $p$-Lucas property.
\end{enumerate}

Furthermore, if $\mathcal{Q}_{e,f}$ satisfies the $p$-Lucas property for all primes $p$, then, for all $\mathbf{n}$ in $\mathbb{N}^d$ and every prime $p$, we have
$$
\mathcal{Q}_{e,f}(\mathbf{n})\in p^{\alpha_p(\mathcal{Q}_{e,f},\mathbf{n})}\mathbb{Z}.
$$
\end{theo}

\begin{Remark}
Theorem \ref{FactoCrit} implies that $\mathcal{Q}_{e,f}$ satisfies the $p$-Lucas property for all primes $p$ if and only if all Taylor coefficients at the origin of the associated mirror maps $z_{e,f,k}$, $1\leq k\leq d$, are integers (see Theorems $1$ and $3$ in \cite{Delaygue3}). Indeed, if $\Delta_{e,f}$ is nonnegative on $[0,1]^d$ and if $|e|\neq|f|$, then there exists $k$ in $\{1,\dots,d\}$ such that $|e|^{(k)}>|f|^{(k)}$.
\end{Remark}

Coster proved in \cite{Coster} similar results to Theorems \ref{theo gene}-\ref{FactoCrit} for the coefficients of certain algebraic power series. Namely, given a prime $p\geq 3$, integers $a_1,\dots,a_{p-1}$, and a sequence $A$ such that
$$
f_A(z)=(1+a_1z+\cdots+a_{p-1}z^{p-1})^{\frac{1}{1-p}},
$$ 
Coster proved that, for all natural integers $n$, we have 
$$
v_p\big(A(n)\big)\geq \left\lfloor\frac{\alpha_p(A,n)+1}{2}\right\rfloor.
$$

\subsection{Auxiliary results}

The proof of Theorem \ref{theo gene} rests on three results which may be useful to study other sequences.
\begin{propo}\label{propo Ap}
Let $p$ be a fixed prime and $A$ a $\mathbb{Z}_p$-valued sequence satisfying the $p$-Lucas property with $A(0)$ in $\mathbb{Z}_p^\times$. Let $\mathfrak{A}$ be the $\mathbb{Z}_p$-module spanned by $A$. Assume that
\begin{itemize}
\item[$(a)$] there exists a set $\mathfrak{B}$ of $\mathbb{Z}_p$-valued sequences with $\mathfrak{A}\subset\mathfrak{B}$ such that, for all $B$ in $\mathfrak{B}$, all $v$ in $\{0,\dots,p-1\}$ and all positive integers $n$, there exist $A'$ in $\mathfrak{A}$ and a sequence $(B_k)_{k\geq 0}$, $B_k$ in $\mathfrak{B}$, such that
$$
B(v+np)=A'(n)+\sum_{k=0}^\infty p^{k+1}B_k(n-k);
$$
\item[$(b)$] $f_A(z)$ is annihilated by a differential operator $\mathcal{L}$ in $\mathbb{Z}_p[z,\theta]$ such that at least one of the following conditions holds:
\begin{itemize}
\item $\mathcal{L}$ is of type I. 
\item $\mathcal{L}$ is of type II and $p-1\in\mathcal{Z}_p(A)$.
\end{itemize}
\end{itemize}
Then, for all $B$ in $\mathfrak{B}$ and all natural integers $n$, we have
$$
A(n)\in p^{\alpha_p(A,n)}\mathbb{Z}_p\quad\textup{and}\quad B(n)\in p^{\alpha_p(A,n)-1}\mathbb{Z}_p.
$$
\end{propo}

In Proposition \ref{propo Ap} and throughout this article, if $(A(n))_{n\geq 0}$ is a sequence taking its values in $\mathbb{Z}$ or $\mathbb{Z}_p$, then, for all negative integers $n$, we set $A(n):=0$.
Therefore, to prove Theorem~\ref{theo gene}, it suffices to show that $\mathfrak{S}_{e,f}$ satisfies the $p$-Lucas property and Condition~$(a)$ of Proposition \ref{propo Ap} with $\mathfrak{B}=\{\mathfrak{S}_{e,f}^g\,:\,g\in\mathfrak{F}_p^d\}$. To that purpose, we shall prove the following results.

\begin{propo}\label{propo Lucas}
Let $e$ and $f$ be disjoint tuples of vectors in $\mathbb{N}^d$ such that $|e|=|f|$ and, for all $\mathbf{x}$ in $\mathcal{D}_{e,f}$, $\Delta_{e,f}(\mathbf{x})\geq 1$. Assume that $e$ is $1$-admissible. Then, $\mathfrak{S}_{e,f}$ is integer-valued and satisfies the $p$-Lucas property for all primes $p$.
\end{propo}

\begin{propo}\label{propo Gamma}
Let $p$ be a fixed prime. We write $\Gamma_p$ for the $p$-adic Gamma function. Then, there exists a function $g$ in $\mathfrak{F}_p^2$ such that, for all natural integers $n$ and $m$, we have
$$
\frac{\Gamma_p\big((m+n)p\big)}{\Gamma_p(mp)\Gamma_p(np)}=1+g(m,n)p.
$$
\end{propo}

\subsection{Application of Theorem \ref{theo gene}}\label{Application}

By applying Theorem \ref{theo gene}, we obtain similar results to Conjectures A-C for numbers satisfying Ap\'ery-like recurrence relations which we list below. Characters in brackets in the last column of the following table form the sequence number in the Online Encyclopedia of Integer Sequences \cite{OEIS}. 
\begin{small}
$$
{\renewcommand{\arraystretch}{2.5}
\begin{array}{|c|c|c|c|}
  \hline
  \textup{Sequence} & \mathcal{Q}_{e,f}(n_1,n_2) & \mathcal{L} & \textup{Reference}\\
  \hline
  \displaystyle{\sum_{k=0}^n\binom{n}{k}^2\binom{n+k}{k}^2} & \displaystyle{\frac{(2n_1+n_2)!^2}{n_1!^4n_2!^2}} & \textup{\cite[$(\gamma)$]{Zudilin}} & \textup{Ap\'ery numbers (A005259)} \\\hline
  \displaystyle{\sum_{k=0}^n\binom{n}{k}^2\binom{n+k}{k}} & \displaystyle{\frac{(2n_1+n_2)!(n_1+n_2)!}{n_1!^3n_2!^2}} & \textup{\cite[\textbf{D}]{Zagier}} & \textup{Ap\'ery numbers (A005258)} \\\hline
  \displaystyle{\binom{2n}{n}=\sum_{k=0}^n\binom{n}{k}^2} & \displaystyle{\frac{(n_1+n_2)!^2}{n_1!^2n_2!^2}} & \textup{type I} & \textup{\parbox{5cm}{\centering Central binomial\\coefficients (A000984)}} \\\hline
  \displaystyle{\sum_{k=0}^n\binom{n}{k}^3} & \displaystyle{\frac{(n_1+n_2)!^3}{n_1!^3n_2!^3}} & \textup{\cite[\textbf{A}]{Zagier}} & \textup{Franel numbers (A000172)} \\\hline
  \displaystyle{\sum_{k=0}^n\binom{n}{k}^4} & \displaystyle{\frac{(n_1+n_2)!^4}{n_1!^4n_2!^4}} & \textup{\cite{Franel1},\cite{Franel2}} & \textup{(A005260)} \\\hline
  \displaystyle{\sum_{k=0}^n\binom{n}{k}\binom{2k}{k}\binom{2(n-k)}{n-k}} & \displaystyle{\frac{(n_1+n_2)!(2n_1)!(2n_2)!}{n_1!^3n_2!^3}} & \textup{\cite[(d)]{Zudilin}}& \textup{(A081085)} \\\hline
  \displaystyle{\sum_{k=0}^n\binom{n}{k}^2\binom{2k}{k}} & \displaystyle{\frac{(n_1+n_2)!^2(2n_1)!}{n_1!^4n_2!^2}} & \textup{\cite[\textbf{C}]{Zagier}} & \textup{\parbox{5cm}{\centering Number of abelian squares \\ of length $2n$ over an alphabet \\ with $3$ letters (A002893)}} \\\hline
  \displaystyle{\sum_{k=0}^n\binom{n}{k}^2\binom{2k}{k}\binom{2(n-k)}{n-k}} & \displaystyle{\frac{(n_1+n_2)!^2(2n_1)!(2n_2)!}{n_1!^4n_2!^4}} & \textup{\cite[($\alpha$)]{Zudilin}} & \textup{Domb numbers (A002895)} \\\hline
  \displaystyle{\sum_{k=0}^n\binom{2k}{k}^2\binom{2(n-k)}{n-k}^2} & \displaystyle{\frac{(2n_1)!^2(2n_2)!^2}{n_1!^4n_2!^4}} & \textup{\cite[($\beta$)]{Zudilin}} & \textup{(A036917)} \\
  \hline
\end{array}
}
$$
\end{small}

All differential operators listed in the above table are of type I for all primes $p$, except the one associated with $A_5(n):=\sum_{k=0}^n\binom{n}{k}^4$ which reads
$$
\mathcal{L}_5=\theta^3-z2(2\theta+1)(3\theta^2+3\theta+1)-z^24(\theta+1)(4\theta+5)(4\theta+3).
$$
Hence $\mathcal{L}_5$ is of type II for all primes $p$. By a result of Calkin \cite[Proposition $3$]{Calkin}, for all primes $p$, we have $A_5(p-1)\equiv 0\mod p$, \textit{i.\,e.} $p-1$ is in $\mathcal{Z}_p(A_5)$. Thus we can apply Theorem~\ref{theo gene} to $A_5$. 

Observe that the generating function of the central binomial coefficients is annihilated by the differential operator $\mathcal{L}=\theta-z(4\theta+2)$ which is of type I for all primes $p$.

According to the recurrence relation found by Almkvist and Zudilin (see Case (d) in \cite{Zudilin}), $A_6(n):=\sum_{k=0}^n\binom{n}{k}\binom{2k}{k}\binom{2(n-k)}{n-k}$ is also Sequence \textbf{E} in Zagier's list \cite{Zagier}, that is
$$
A_6(n)=\sum_{k=0}^{\lfloor n/2\rfloor}4^{n-2k}\binom{n}{2k}\binom{2k}{k}^2.
$$

Furthermore, according to \cite{Shallit}, Domb numbers $A_8(n)=\sum_{k=0}^n\binom{n}{k}^2\binom{2k}{k}\binom{2(n-k)}{n-k}$ are also the numbers of abelian squares of length $2n$ over an alphabet with $4$ letters. 
\medskip

Now we consider the numbers $C_i(n)$ of abelian squares of length $2n$ over an alphabet with $i$ letters which, for all positive integers $i\geq 2$, satisfy (see \cite{Shallit})
$$
C_i(n)=\underset{k_1,\dots,k_i\in\mathbb{N}}{\sum_{k_1+\cdots +k_i=n}}\left(\frac{n!}{k_1!\cdots k_i!}\right)^2.
$$
According to \cite{RandomWalks}, $C_i(n)$ is also the $2n$-th moment of the distance to the origin after $i$ steps traveled by a walk in the plane with unit steps in random directions. 

To apply Theorem \ref{theo gene} to $C_i$, it suffices to show that $f_{C_i}$ is annihilated by a differential operator of type I for all primes $p$. Indeed, by Proposition $1$ and Theorem $2$ in \cite{RandomWalks}, for all $j\geq 2$, $C_j(n)$ satisfies the recurrence relation of order $\lceil j/2\rceil$ with polynomial coefficients of degree $j-1$:
\begin{equation}\label{RecWalks}
n^{j-1}C_j(n)+\sum_{i\geq 1}\left(n^{j-1}\sum_{\alpha_1,\dots,\alpha_i}\prod_{k=1}^i(-\alpha_k)(j+1-\alpha_k)\left(\frac{n-k}{n-k+1}\right)^{\alpha_k-1}\right)C_j(n-i)=0,
\end{equation}
where the sum is over all sequences of positive integers $\alpha_1,\dots,\alpha_i$ satisfying $\alpha_k\leq j$ and $\alpha_{k+1}\leq\alpha_k-2$. We consider $i\geq 2$ and $i$ positive integers $\alpha_1,\dots,\alpha_i\leq j$ satisfying $\alpha_{k+1}\leq\alpha_k-2$. We have
$$
n^{j-1}\prod_{k=1}^i\left(\frac{n-k}{n-k+1}\right)^{\alpha_k-1}=\frac{n^{j-1}}{n^{\alpha_1-1}}\left(\prod_{k=1}^{i-1}(n-k)^{\alpha_k-\alpha_{k+1}}\right)(n-i)^{\alpha_i-1},
$$
with $j-\alpha_1\geq 0$, $\alpha_k-\alpha_{k+1}\geq 2$ and $\alpha_i-1\geq 0$. Then, $f_{C_j}(z)$ is annihilated by a differential operator $\mathcal{L}=P_0(\theta)+zP_1(\theta)+\cdots+z^qP_q(\theta)$ with $P_0(\theta)=\theta^{j-1}$ and, for all $i\geq 2$,
$$
P_i(\theta)\in\prod_{k=1}^{i-1}(\theta+i-k)^2\mathbb{Z}[\theta]\subset\prod_{k=1}^{i-1}(\theta+k)^2\mathbb{Z}[\theta],
$$
so that $\mathcal{L}$ is of type I for all primes $p$, as expected.

\subsection{Structure of the article}

In Section \ref{section theo 3}, we use several results of \cite{Delaygue3} to prove Theorem~\ref{FactoCrit}. Section \ref{section rec} is devoted to the proofs of Theorem \ref{Laurent Ap} and Proposition \ref{propo Ap}. In particular, we prove Lemma \ref{lemme induction} which points out the role played by differential operators in our proofs. In Section \ref{section main proof}, we prove Theorem \ref{theo gene} by applying Proposition \ref{propo Ap} to $\mathfrak{S}_{e,f}$. It is the most technical part of this article.

\section{Proof of Theorem \ref{FactoCrit}}\label{section theo 3}

First, we prove that if $|e|=|f|$, then, for all primes $p$, all $\mathbf{a}$ in $\{0,\dots,p-1\}^d$ and all $\mathbf{n}$ in $\mathbb{N}^d$, we have
\begin{equation}\label{inter1}
\frac{\mathcal{Q}_{e,f}(\mathbf{a}+\mathbf{n}p)}{\mathcal{Q}_{e,f}(\mathbf{a})\mathcal{Q}_{e,f}(\mathbf{n})}\in\frac{\prod_{i=1}^u\prod_{j=1}^{\lfloor \mathbf{e}_i\cdot\mathbf{a}/p\rfloor}\left(1+\frac{\mathbf{e}_i\cdot\mathbf{n}}{j}\right)}{\prod_{i=1}^v\prod_{j=1}^{\lfloor \mathbf{f}_i\cdot\mathbf{a}/p\rfloor}\left(1+\frac{\mathbf{f}_i\cdot\mathbf{n}}{j}\right)}(1+p\mathbb{Z}_p).
\end{equation}
Indeed, we have
$$
\frac{\mathcal{Q}_{e,f}(\mathbf{a}+\mathbf{n}p)}{\mathcal{Q}_{e,f}(\mathbf{a})\mathcal{Q}_{e,f}(\mathbf{n})}=\frac{\mathcal{Q}_{e,f}(\mathbf{a}+\mathbf{n}p)}{\mathcal{Q}_{e,f}(\mathbf{a})\mathcal{Q}_{e,f}(\mathbf{n}p)}\cdot\frac{\mathcal{Q}_{e,f}(\mathbf{n}p)}{\mathcal{Q}_{e,f}(\mathbf{n})}.
$$
Since $|e|=|f|$, we can apply \cite[Lemma $7$]{Delaygue3} (\footnote{The proof of this lemma uses a lemma of Lang which contains an error. Fortunately, Lemma $7$ remains true. Details of this correction are presented in \cite[Section $2.4$]{Delaygue4}.}) with $\mathbf{c}=\mathbf{0}$, $\mathbf{m}=\mathbf{n}$ and $s=0$ which yields
$$
\frac{\mathcal{Q}_{e,f}(\mathbf{n}p)}{\mathcal{Q}_{e,f}(\mathbf{n})}\in 1+p\mathbb{Z}_p.
$$
Furthermore, we have
\begin{align*}
\frac{\mathcal{Q}_{e,f}(\mathbf{a}+\mathbf{n}p)}{\mathcal{Q}_{e,f}(\mathbf{a})\mathcal{Q}_{e,f}(\mathbf{n}p)}&=\frac{1}{\mathcal{Q}_{e,f}(\mathbf{a})}\frac{\prod_{i=1}^u\prod_{j=1}^{\mathbf{e}_i\cdot\mathbf{a}}(j+\mathbf{e}_i\cdot\mathbf{n}p)}{\prod_{i=1}^v\prod_{j=1}^{\mathbf{f}_i\cdot\mathbf{a}}(j+\mathbf{f}_i\cdot\mathbf{n}p)}\\
&=\frac{\prod_{i=1}^u\prod_{j=1}^{\mathbf{e}_i\cdot\mathbf{a}}\left(1+\frac{\mathbf{e}_i\cdot\mathbf{n}p}{j}\right)}{\prod_{i=1}^v\prod_{j=1}^{\mathbf{f}_i\cdot\mathbf{a}}\left(1+\frac{\mathbf{f}_i\cdot\mathbf{n}p}{j}\right)}\\
&\in\frac{\prod_{i=1}^u\prod_{j=1}^{\lfloor \mathbf{e}_i\cdot\mathbf{a}/p\rfloor}\left(1+\frac{\mathbf{e}_i\cdot\mathbf{n}}{j}\right)}{\prod_{i=1}^v\prod_{j=1}^{\lfloor \mathbf{f}_i\cdot\mathbf{a}/p\rfloor}\left(1+\frac{\mathbf{f}_i\cdot\mathbf{n}}{j}\right)}(1+p\mathbb{Z}_p),
\end{align*}
because, if $p$ does not divide $j$, then $1+(\mathbf{e}_i\cdot\mathbf{n}p)/j$ belongs to $1+p\mathbb{Z}_p$. This finishes the proof of \eqref{inter1}.
\medskip

Now we prove Assertion $(1)$ in Theorem \ref{FactoCrit}. Let $p$ be a fixed prime number. It is well known that, for all natural integers $n$, we have
$$
v_p(n!)=\sum_{\ell=1}^\infty\left\lfloor\frac{n}{p^\ell}\right\rfloor.
$$
Thus, for all vectors $\mathbf{n}$ in $\mathbb{N}^d$, we have 
$$
v_p\big(\mathcal{Q}_{e,f}(\mathbf{n})\big)=\sum_{\ell=1}^\infty\Delta_{e,f}\left(\frac{\mathbf{n}}{p^{\ell}}\right).
$$ 
Let fix $\mathbf{n}$ in $\mathbb{N}^d$ and $\mathbf{a}$ in $\{0,\dots,p-1\}^d$. Let $\{\cdot\}$ denote the fractional part function. For any vector of real numbers $\mathbf{x}=(x_1,\dots,x_d)$, we set $\{\mathbf{x}\}:=(\{x_1\},\dots,\{x_d\})$. Since $|e|=|f|$, we have 
$$
v_p\big(\mathcal{Q}_{e,f}(\mathbf{a}+\mathbf{n}p)\big)=\sum_{\ell=1}^\infty\Delta_{e,f}\left(\left\{\frac{\mathbf{a}+\mathbf{n}p}{p^\ell}\right\}\right)\geq\Delta_{e,f}\left(\frac{\mathbf{a}}{p}\right),
$$
because $\Delta_{e,f}$ is nonnegative on $[0,1]^d$. On the one hand, if $\mathbf{a}/p$ is in $\mathcal{D}_{e,f}$, then both $\mathcal{Q}_{e,f}(\mathbf{a}+\mathbf{n}p)$ and $\mathcal{Q}_{e,f}(\mathbf{a})\mathcal{Q}_{e,f}(\mathbf{n})$ are congruent to $0$ modulo $p$. On the other hand, if $\mathbf{a}/p$ is not in $\mathcal{D}_{e,f}$, then, for all $\mathbf{d}$ in $e$ or $f$, we have $\lfloor \mathbf{d}\cdot\mathbf{a}/p\rfloor=0$ so that \eqref{inter1} yields
$$
\mathcal{Q}_{e,f}(\mathbf{a}+\mathbf{n}p)\equiv\mathcal{Q}_{e,f}(\mathbf{a})\mathcal{Q}_{e,f}(\mathbf{n})\mod p\mathbb{Z}_p,
$$
as expected. This proves Assertion $(1)$ in Theorem \ref{FactoCrit}.
\medskip 

Now we prove Assertion $(2)$ in Theorem \ref{FactoCrit}. If $|e|\neq|f|$ then, since $\Delta_{e,f}$ is nonnegative on $[0,1]^d$, there exists $k$ in $\{1,\dots,d\}$ such that $|e|^{(k)}-|f|^{(k)}=\Delta_{e,f}(\mathbf{1}_k)\geq 1$. Thereby, for almost all primes $p$, we have 
$$
v_p\big(\mathcal{Q}_{e,f}(\mathbf{1}_k+\mathbf{1}_kp)\big)=\sum_{\ell=1}^\infty\Delta_{e,f}\left(\frac{\mathbf{1}_k+\mathbf{1}_kp}{p^\ell}\right)\geq\Delta_{e,f}\left(\frac{\mathbf{1}_k}{p}+\mathbf{1}_k\right)\geq 1,
$$
but $v_p\big(\mathcal{Q}_{e,f}(\mathbf{1}_k)\big)=0$ so that $\mathcal{Q}_{e,f}$ does not satisfy the $p$-Lucas property.
\medskip

Throughout the rest of this proof, we assume that $|e|=|f|$. According to Section $7.3.2$ in \cite{Delaygue3}, there exist $k$ in $\{1,\dots,d\}$ and a rational fraction $R(X)$ in $\mathbb{Q}(X)$, $R(X)\neq 1$, such that, for all large enough prime numbers $p$, we can choose $\mathbf{a}_p$ in $\{0,\dots,p-1\}^d$ satisfying $\mathcal{Q}_{e,f}(\mathbf{a}_p)\in\mathbb{Z}_p^\times$, and such that, for all natural integers $n$, we have (see \cite[$(7.10)$]{Delaygue3})
$$
\mathcal{Q}_{e,f}(\mathbf{a}_p+\mathbf{1}_knp)\in R(n)\mathcal{Q}_{e,f}(\mathbf{a}_p)\mathcal{Q}_{e,f}(\mathbf{1}_kn)(1+p\mathbb{Z}_p).
$$
We fix a natural integer $n$ satisfying $R(n)\neq 1$. For almost all primes $p$, the numbers $R(n)$, $\mathcal{Q}_{e,f}(\mathbf{1}_kn)$ and $\mathcal{Q}_{e,f}(\mathbf{a}_p)$ are invertible in $\mathbb{Z}_p$, and $R(n)\not\equiv 1\mod p\mathbb{Z}_p$. Thus, we obtain 
$$
\mathcal{Q}_{e,f}(\mathbf{a}_p+\mathbf{1}_knp)\not\equiv \mathcal{Q}_{e,f}(\mathbf{a}_p)\mathcal{Q}_{e,f}(\mathbf{1}_kn)\mod p\mathbb{Z}_p,
$$
which finishes the proof of Assertion $(2)$ in Theorem \ref{FactoCrit}.
\medskip

Now we assume that $|e|=|f|$ and that, for all $\mathbf{x}$ in $\mathcal{D}_{e,f}$, we have $\Delta_{e,f}(\mathbf{x})\geq 1$. Hence, for every prime $p$, we have 
$$
\mathcal{Z}_p(\mathcal{Q}_{e,f})=\big\{\mathbf{v}\in\{0,\dots,p-1\}^d\,:\,\mathbf{v}/p\in\mathcal{D}_{e,f}\big\}.
$$
Furthermore, if $\mathbf{v}/p$ belongs to $\mathcal{D}_{e,f}$, then, for all positive integers $N$ and all vectors $\mathbf{a}_0,\dots,\mathbf{a}_{N-1}$ in $\{0,\dots,p-1\}^d$, we have 
$$
\frac{\mathbf{v}}{p}\leq\left\{\frac{\mathbf{a}_0+\mathbf{a}_1p+\cdots+\mathbf{a}_{N-1}p^{N-1}+\mathbf{v}p^N}{p^{N+1}}\right\}\in\mathcal{D}_{e,f},
$$
so that, for every $\mathbf{n}$ in $\mathbb{N}^d$, $\mathbf{n}=\sum_{k=0}^\infty\mathbf{n}_kp^k$ with $\mathbf{n}_k\in\{0,\dots,p-1\}^d$, we have
$$
v_p\big(\mathcal{Q}_{e,f}(\mathbf{n})\big)=\sum_{\ell=1}^\infty\Delta_{e,f}\left(\left\{\frac{\sum_{k=0}^{\ell-1}\mathbf{n}_kp^k}{p^{\ell}}\right\}\right)\geq\alpha_p(\mathcal{Q}_{e,f},\mathbf{n}),
$$
and Theorem \ref{FactoCrit} is proved.

\section{Proofs of Theorem \ref{Laurent Ap} and Proposition \ref{propo Ap}}\label{section rec}

\subsection{Induction \textit{via} Ap\'ery-like recurrence relations}

In this section, we fix a prime $p$. If $A$ is a $\mathbb{Z}_p$-valued sequence, then, for all natural integers $r$, we write $\mathcal{U}_A(r)$ for the assertion ``For all $n,i\in\mathbb{N}$, $i\leq r$, if $\alpha_p(A,n)\geq i$, then $A(n)\in p^i\mathbb{Z}_p$''. As a first step, we shall prove the following result.

\begin{lemme}\label{lemme induction}
Let $A$ be a $\mathbb{Z}_p$-valued sequence satisfying the $p$-Lucas property with $A(0)$ in $\mathbb{Z}_p^\times$. Assume that $f_A$ is annihilated by a differential operator $\mathcal{L}\in\mathbb{Z}_p[z,\theta]$ such that at least one of the following conditions holds:
\begin{itemize}
\item $\mathcal{L}$ is of type I. 
\item $\mathcal{L}$ is of type II and $p-1\in\mathcal{Z}_p(A)$. 
\end{itemize}
Let $r$ be a natural integer such that $\mathcal{U}_A(r)$ holds. Then, for all $n_0$ in $\mathcal{Z}_p(A)$ and all natural integers $m$ satisfying $\alpha_p(A,m)\geq r$, we have
$$
A(n_0+mp)\in p^{r+1}\mathbb{Z}_p.
$$
\end{lemme}

\begin{proof}
Since $A$ satisfies the $p$-Lucas property, we can assume that $r$ is nonzero. The series $f_A(z)$ is annihilated by a differential operator $\mathcal{L}=P_0(\theta)+zP_1(\theta)+\cdots+z^qP_q(\theta)$ with $P_k(X)$ in $\mathbb{Z}_p[X]$ and $P_0(\mathbb{Z}_p^\times)\subset\mathbb{Z}_p^\times$. Thus, for every natural integer $n$, we have
\begin{equation}\label{eq rec}
\sum_{k=0}^qP_k(n-k)A(n-k)=0.
\end{equation}

We fix a natural integer $m$ satisfying $\alpha_p(A,m)\geq r$. In particular, since $r$ is nonzero and $A(0)$ is invertible in $\mathbb{Z}_p$, we have $m\geq 1$. Furthermore, for all $v$ in $\{0,\dots,p-1\}$, we also have $\alpha_p(A,v+mp)\geq r$. According to $\mathcal{U}_A(r)$, we obtain that, for all $v$ in $\{0,\dots,p-1\}$, $A(v+mp)$ belongs to $p^{r}\mathbb{Z}_p$ so that $A(v+mp)=:\beta(v,m)p^{r}$, with $\beta(v,m)\in\mathbb{Z}_p$. 

By \eqref{eq rec}, for all $v$ in $\{q,\dots,p-1\}$, we have
\begin{align*}
0=\sum_{k=0}^qP_k(v-k+mp)A(v-k+mp)&=p^{r}\sum_{k=0}^qP_k(v-k+mp)\beta(v-k,m)\\
&\equiv p^{r}\sum_{k=0}^qP_k(v-k)\beta(v-k,m)\mod p^{r+1}\mathbb{Z}_p,
\end{align*}
because, for all polynomials $P$ in $\mathbb{Z}_p[X]$ and all integers $a$ and $c$, we have $P(a+cp)\equiv P(a)\mod p\mathbb{Z}_p$.
Thus, for all $v$ in $\{q,\dots,p-1\}$, we obtain
\begin{equation}\label{rec mod p}
\sum_{k=0}^qP_k(v-k)\beta(v-k,m)\equiv 0\mod p\mathbb{Z}_p.
\end{equation}

We claim that if $v$ is in $\{1,\dots,q-1\}$, then, for all $k$ in $\{v+1,\dots,q\}$, we have 
\begin{equation}\label{rescue}
P_k(v+mp-k)A(v+mp-k)\in p^{r+1}\mathbb{Z}_p.
\end{equation}

Indeed, on the one hand, if $\mathcal{L}$ is of type II, then we have $q=2$ and $P_2(X)$ belongs to $(X+1)\mathbb{Z}_p[X]$ which yields
$$
P_2(-1+mp)A(-1+mp)\in pA\big(p-1+(m-1)p\big)\mathbb{Z}_p.
$$
Since $0$ is not in $\mathcal{Z}_p(A)$, we have $\alpha_p(A,m-1)\geq r-1$ which, together with $p-1\in\mathcal{Z}_p(A)$, leads to 
$$
\alpha_p\big(A,p-1+(m-1)p\big)\geq r. 
$$
According to $\mathcal{U}_A(r)$, we obtain that $pA\big(p-1+(m-1)p\big)$ is in $p^{r+1}\mathbb{Z}_p$, as expected. On the other hand, if $\mathcal{L}$ is of type I, then for all $v$ in $\{1,\dots,q-1\}$ and all $k$ in $\{v+1,\dots,q\}$, we have
$$
v_p\big(P_k(v+mp-k)\big)\geq v_p\left(\prod_{i=1}^{k-1}(v+mp-k+i)^2\right).
$$
Writing $k-v=a+bp$ with $a$ in $\{0,\dots,p-1\}$ and $b$ in $\mathbb{N}$, we obtain $k-1\geq a+bp$ so that
$$
v_p\left(\prod_{i=1}^{k-1}(mp+i-a-bp)\right)\geq\begin{cases}b & \textup{if $a=0$;}\\ b+1 & \textup{if $a\geq 1$.}\end{cases}.
$$
Thus, it is enough to prove that 
\begin{equation}\label{Descente}
A(v+mp-k)\in \begin{cases} p^{r+1-2b}\mathbb{Z}_p & \textup{if $a=0$;}\\ p^{r-1-2b}\mathbb{Z}_p & \textup{if $a\geq 1$}.\end{cases}.
\end{equation}
We have $v+mp-k=-a+(m-b)p$. If $-a+(m-b)p$ is negative, then $A(v+mp-k)=0$ and \eqref{Descente} holds. If $m-b$ is nonnegative, then we have $\alpha_p(A,m-b)\geq r-b$. Thus, we have either $a=0$ and $\alpha_p(A,v+mp-k)\geq r-b$, or $a,m-b\geq 1$ and 
$$
\alpha_p(A,v+mp-k)=\alpha_p\big(A,p-a+(m-b-1)p\big)\geq r-b-1.
$$
Hence Assertion $\mathcal{U}_A(r)$ yields 
$$
A(v+mp-k)\in\begin{cases} p^{r-b}\mathbb{Z}_p & \textup{if $a=0$;}\\ p^{r-1-b}\mathbb{Z}_p & \textup{if $a\geq 1$.}\end{cases}.
$$
If $a=0$, then $b\geq 1$ so that \eqref{Descente} holds and \eqref{rescue} is proved. 
\medskip

By \eqref{rescue}, for all natural integers $v$ satisfying $1\leq v\leq\min(q-1,p-1)$, we have
\begin{align*}
0&=\sum_{k=0}^qP_k(v-k+mp)A(v-k+mp)\\
&\equiv\sum_{k=0}^vP_k(v-k+mp)A(v-k+mp)\mod p^{r+1}\mathbb{Z}_p\\
&\equiv p^r\sum_{k=0}^vP_k(v-k+mp)\beta(v-k,m)\mod p^{r+1}\mathbb{Z}_p\\
&\equiv p^r\sum_{k=0}^vP_k(v-k)\beta(v-k,m)\mod p^{r+1}\mathbb{Z}_p.
\end{align*}
Thus, for all natural integers $v$ satisfying $1\leq v\leq\min(q-1,p-1)$, we have
\begin{equation}\label{rec mod p 0}
\sum_{k=0}^vP_k(v-k)\beta(v-k,m)\equiv 0\mod p\mathbb{Z}_p.
\end{equation}
\medskip 

Both sequences $(\beta(v,m))_{0\leq v\leq p-1}$ and $(A(v))_{0\leq v\leq p-1}$ satisfy Equations \eqref{rec mod p} and \eqref{rec mod p 0}. Furthermore, for all $v$ in $\{1,\dots,p-1\}$, $P_0(v)$ and $A(0)$ are invertible in $\mathbb{Z}_p$. Hence there exists $\gamma(m)$ in $\{0,\dots,p-1\}$ such that, for all $v$ in $\{0,\dots,p-1\}$, we have $\beta(v,m)\equiv A(v)\gamma(m)\mod p\mathbb{Z}_p$ so that
$$
A(v+mp)\equiv A(v)\gamma(m)p^{r}\mod p^{r+1}\mathbb{Z}_p.
$$
Since $n_0$ is in $\mathcal{Z}_p(A)$, we obtain that $A(n_0+mp)$ belongs to $p^{r+1}\mathbb{Z}_p$ and Lemma \ref{lemme induction} is proved.
\end{proof}

\subsection{Proof of Theorem \ref{Laurent Ap}}

Let $p$ be a fixed prime number. For every positive integer $n$, we set $\ell(n):=\lfloor\log_p(n)\rfloor+1$ the length of the expansion of $n$ to the base $p$, and $\ell(0):=1$. For all natural integers $n_1,\dots,n_r$, we set
$$
n_1\ast\cdots\ast n_r:=n_1+n_2p^{\ell(n_1)}+\cdots+n_r p^{\ell(n_1)+\cdots+\ell(n_{r-1})},
$$
so that the expansion of $n_1\ast\cdots\ast n_r$ to the base $p$ is the concatenation of the respective expansions of $n_1,\dots,n_r$. Then, by a result of Mellit and Vlasenko \cite[Lemma $1$]{Masha}, there exists a $\mathbb{Z}_p$-valued sequence $(c_n)_{n\geq 0}$ such that, for all positive integers $n$, we have
\begin{equation}\label{Decomp}
A(n)=\underset{1\leq r\leq \ell(n),\;n_r> 0}{\sum_{n_1\ast\cdots\ast n_r=n}}c_{n_1}\cdots c_{n_r}\quad\textup{and}\quad c_n\equiv 0\mod p^{\ell(n)-1}\mathbb{Z}_p. 
\end{equation}

For every natural integer $r$, we write $\mathcal{U}(r)$ for the assertion: ``For all $n,i\in\mathbb{N}$, $i\leq r$, if $\alpha_p(A,n)\geq i$, then $A(n),c_n\in p^i\mathbb{Z}_p$''. To prove Theorem \ref{Laurent Ap}, it suffices to show that, for all natural integers $r$, Assertion $\mathcal{U}(r)$ holds. 
\medskip

First we prove $\mathcal{U}(1)$. By Theorem $1$ in \cite{Masha}, $A$ satisfies the $p$-Lucas property. In addition, if $v$ is in $\mathcal{Z}_p(A)$, then $v$ is nonzero because $A(0)=1$, and by \eqref{Decomp} we have $c_v=A(v)\in p\mathbb{Z}_p$. Now, if a natural integer $n$ satisfies $\ell(n)=2$ and $\alpha_p(A,n)\geq 1$, then Equation~\eqref{Decomp} yields $A(n)\equiv c_n\mod p\mathbb{Z}_p$, so that $c_n$ is in $p\mathbb{Z}_p$. Hence, by induction on $\ell(n)$, we obtain that, for all natural integers $n$ satisfying $\alpha_p(A,n)\geq 1$, $c_n$ belongs to $p\mathbb{Z}_p$, so that $\mathcal{U}(1)$ holds.
\medskip

Let $r$ be a positive integer such that $\mathcal{U}(r)$ holds. We shall prove that $\mathcal{U}(r+1)$ is true. For all positive integers $M$, we write $\mathcal{U}_M(r+1)$ for the assertion: 
\begin{quote}
``For all $n,i\in\mathbb{N}$, $n\leq M$, $i\leq r+1$, if $\alpha_p(A,n)\geq i$, then $A(n),c_n\in p^{i}\mathbb{Z}_p$''. 
\end{quote}
Hence $\mathcal{U}_M(r+1)$ is true if $\ell(M)\leq r$. Let $M$ be a positive integer such that $\mathcal{U}_M(r+1)$ holds. We shall prove $\mathcal{U}_{M+1}(r+1)$. By Assertions $\mathcal{U}(r)$ and $\mathcal{U}_M(r+1)$, it suffices to prove that if $\alpha_p(A,M+1)$ is greater than $r$, then $A(M+1)$ and $c_{M+1}$ belong to $p^{r+1}\mathbb{Z}_p$. In the rest of the proof, we assume that $\alpha_p(A,M+1)$ is greater than $r$. 
\medskip

If $u$ and $n_1,\dots,n_u$ are natural integers satisfying $2\leq u\leq\ell(M+1)$ and $n_1\ast\cdots\ast n_u=M+1$ with $n_u>0$, then, for all $i$ in $\{1,\dots,u\}$, we have $n_i\leq M$ and 
$$
\alpha_p(A,n_1)+\cdots+\alpha_p(A,n_u)=\alpha_p(A,M+1)\geq r+1. 
$$
Then there exist $1\leq a_1<\cdots<a_k\leq u$ and $1\leq i_1,\dots,i_k\leq r+1$ such that $\alpha_p(A,n_{a_j})\geq i_j$ and $i_1+\cdots+i_k\geq r+1$. Thereby, Assertion $\mathcal{U}_M(r+1)$ yields $c_{n_1}\cdots c_{n_u}\in p^{r+1}\mathbb{Z}_p$, so that
$$
\underset{2\leq u\leq\ell(M+1),\;n_u> 0}{\sum_{n_1\ast\cdots\ast n_u=M+1}}c_{n_1}\cdots c_{n_u}\in p^{r+1}\mathbb{Z}_p.
$$
By \eqref{Decomp}, we obtain
$$
A(M+1)\equiv c_{M+1}\mod p^{r+1}\mathbb{Z}_p\quad\textup{and}\quad c_{M+1}\equiv 0\mod p^{\ell(M+1)-1}\mathbb{Z}_p.
$$
Hence it suffices to consider the case $\ell(M+1)=r+1$. In particular, we have $M+1=v+mp$ where $v$ is in $\mathcal{Z}_p(A)$ and $m$ is a natural integer satisfying $\alpha_p(A,m)=r$. Since $\mathcal{U}(r)$ holds, Lemma \ref{lemme induction} yields $A(M+1)\in p^{r+1}\mathbb{Z}_p$. Thus we also have $c_{M+1}\in p^{r+1}\mathbb{Z}_p$ and Assertion $\mathcal{U}_{M+1}(r+1)$ holds. This finishes the proof of $\mathcal{U}(r+1)$ so that of Theorem \ref{Laurent Ap}.
$\hfill\square$

\subsection{Proof of Proposition \ref{propo Ap}}

Let $p$ be a prime and $A$ a $\mathbb{Z}_p$-valued sequence satisfying hypothesis of Proposition \ref{propo Ap}. For every natural integer $n$, we write $\alpha(n)$, respectively $\mathcal{Z}$, as a shorthand for $\alpha_p(A,n)$, respectively for $\mathcal{Z}_p(A)$. For every natural integer $r$, we define Assertions
$$
\mathcal{U}(r):\textup{ ``For all $n,i\in\mathbb{N}$, $i\leq r$, if $\alpha(n)\geq i$, then $A(n)\in p^{i}\mathbb{Z}_p$.''},
$$
and
$$
\mathcal{V}(r):\textup{ ``For all $n,i\in\mathbb{N}$, $i\leq r$, and all $B\in\mathfrak{B}$, if $\alpha(n)\geq i$, then $B(n)\in p^{i-1}\mathbb{Z}_p$''}.
$$

To prove Proposition \ref{propo Ap}, we have to show that, for all natural integers $r$, Assertions $\mathcal{U}(r)$ and $\mathcal{V}(r)$ are true. We shall prove those assertions by induction on $r$.
\medskip

Observe that Assertions $\mathcal{U}(0)$, $\mathcal{V}(0)$ and $\mathcal{V}(1)$ are trivial. Furthermore, since $A$ satisfies the $p$-Lucas property, Assertion $\mathcal{U}(1)$ holds. Let $r_0$ be a fixed positive integer, $r_0\geq 2$, such that Assertions $\mathcal{U}(r_0-1)$ and $\mathcal{V}(r_0-1)$ are true. First, we prove Assertion $\mathcal{V}(r_0)$.
\medskip

Let $B$ in $\mathfrak{B}$ and $m$ in $\mathbb{N}$ be such that $\alpha(m)\geq r_0$. We write $m=v+np$ with $v$ in $\{0,\dots,p-1\}$. Since $r_0\geq 2$ and $0$ does not belong to $\mathcal{Z}$, we have $n\geq 1$ and, by Assertion~$(a)$ in Proposition~\ref{propo Ap}, there exist $A'$ in $\mathfrak{A}$ and a sequence $(B_k)_{k\geq 0}$, with $B_k$ in $\mathfrak{B}$, such that
\begin{equation}\label{Transf1} 
B(v+np)=A'(n)+\sum_{k=0}^\infty p^{k+1}B_k(n-k).
\end{equation}
In addition, we have $\alpha(n)\geq r_0-1$ and, since $0$ is not in $\mathcal{Z}$, we have $\alpha(n-1)\geq r_0-2$. By induction, for all natural integers $k$ satisfying $k\leq n$, we have $\alpha(n-k)\geq r_0-1-k$. Thus, by \eqref{Transf1} in combination with $\mathcal{U}(r_0-1)$ and $\mathcal{V}(r_0-1)$, we obtain 
$$
A'(n)\in p^{r_0-1}\mathbb{Z}\quad\textup{and}\quad p^{k+1}B_k(n-k)\in p^{k+1+r_0-2-k}\mathbb{Z}_p\subset p^{r_0-1}\mathbb{Z}_p,
$$
so that $B(v+np)$ belongs to $p^{r_0-1}\mathbb{Z}_p$ and $\mathcal{V}(r_0)$ is true.
\medskip

Now we prove Assertion $\mathcal{U}(r_0)$. We write $\mathcal{U}_N(r_0)$ for the assertion: 
\begin{quote}
``For all $n,i\in\mathbb{N}$, $n\leq N$, $i\leq r_0$, if $\alpha(n)\geq i$, then $A(n)\in p^{i}\mathbb{Z}_p$''. 
\end{quote}
We shall prove $\mathcal{U}_N(r_0)$ by induction on $N$. Assertion $\mathcal{U}_1(r_0)$ holds. Let $N$ be a positive integer such that $\mathcal{U}_N(r_0)$ is true. Let $n:=n_0+mp\leq N+1$ with $n_0$ in $\{0,\dots,p-1\}$ and $m$ in $\mathbb{N}$. We can assume that $\alpha(n)\geq r_0$. 

If $n_0$ is in $\mathcal{Z}$, then we have $\alpha(m)\geq r_0-1$ and, by Lemma \ref{lemme induction}, we obtain that $A(n)$ belongs to $p^{r_0}\mathbb{Z}_p$ as expected. If $n_0$ is not in $\mathcal{Z}$, then we have $\alpha(m)\geq r_0$. By Assertion $(a)$ in Proposition \ref{propo Ap}, there exist $A'$ in $\mathfrak{A}$ and a sequence $(B_k)_{k\geq 0}$ with $B_k$ in $\mathfrak{B}$ such that 
$$
A(n)=A'(m)+\sum_{k=0}^\infty p^{k+1}B_k(m-k). 
$$
We have $m\leq N$, $\alpha(m)\geq r_0$ and $\alpha(m-k)\geq r_0-k$, hence, by Assertions $\mathcal{U}_{N}(r_0)$ and $\mathcal{V}(r_0)$, we obtain that $A(n)$ belongs to $p^{r_0}\mathbb{Z}_p$. This finishes the induction on $N$ and proves $\mathcal{U}(r_0)$. Therefore, by induction on $r_0$, Proposition \ref{propo Ap} is proved.
$\hfill\square$

\section{Proof of Theorem \ref{theo gene}}\label{section main proof}

To prove Theorem \ref{theo gene}, we shall apply Proposition \ref{propo Ap} to $\mathfrak{S}_{e,f}$. As a first step, we prove that this sequence satisfies the $p$-Lucas property.

\begin{proof}[Proof of Proposition \ref{propo Lucas}]

For all $\mathbf{x}$ in $[0,1]^d$, we have $\Delta_{e,f}(\mathbf{x})=\Delta_{e,f}(\{\mathbf{x}\})\geq 0$ so that, by Landau's criterion, $\mathcal{Q}_{e,f}$ is integer-valued. Let $p$ be a fixed prime, $v$ in $\{0,\dots,p-1\}$ and $n$ a natural integer. We have
$$
\mathfrak{S}_{e,f}(v+np)=\underset{k_i\in\mathbb{N}}{\sum_{k_1+\cdots+k_d=v+np}}\mathcal{Q}_{e,f}(k_1,\dots,k_d).
$$
Write $k_i=a_i+m_ip$ with $a_i$ in $\{0,\dots,p-1\}$ and $m_i$ in $\mathbb{N}$. If $a_1+\cdots+a_d\neq v$, then we have $a_1+\cdots+a_d\geq p$ and there exists $i$ in $\{1,\dots,d\}$ such that $a_i\geq p/d$. Since $e$ is $1$-admissible, $(a_1,\dots,a_d)/p$ belongs to $\mathcal{D}_{e,f}$ so that $\Delta_{e,f}\big((a_1,\dots,a_p)/p\big)\geq 1$ and $\mathcal{Q}_{e,f}(k_1,\dots,k_d)$ is in $p\mathbb{Z}_p$. In addition, by Theorem \ref{FactoCrit}, $\mathcal{Q}_{e,f}$ satisfies the $p$-Lucas property for all primes $p$. Hence we obtain
\begin{align*}
\mathfrak{S}_{e,f}(v+np)&\equiv\underset{0\leq a_i\leq p-1}{\sum_{a_1+\cdots+a_d=v}}\;\underset{m_i\in\mathbb{N}}{\sum_{m_1+\cdots+m_d=n}}\mathcal{Q}_{e,f}(a_1+m_1p,\dots,a_d+m_dp)\mod p\mathbb{Z}_p\\
&\equiv \underset{0\leq a_i\leq p-1}{\sum_{a_1+\cdots+a_d=v}}\;\underset{m_i\in\mathbb{N}}{\sum_{m_1+\cdots+m_d=n}}\mathcal{Q}_{e,f}(a_1,\dots,a_d)\mathcal{Q}_{e,f}(m_1,\dots,m_d)\mod p\mathbb{Z}_p\\
&\equiv \mathfrak{S}_{e,f}(v)\mathfrak{S}_{e,f}(n)\mod p\mathbb{Z}_p.
\end{align*}
This finishes the proof of Proposition \ref{propo Lucas}.
\end{proof}

If $e$ is $2$-admissible then $e$ is also $1$-admissible. Furthermore, if $f=(\mathbf{1}_{k_1},\dots,\mathbf{1}_{k_v})$, then, for all $\mathbf{x}$ in $\mathcal{D}_{e,f}$, we have
$$
\Delta_{e,f}(\mathbf{x})=\sum_{i=1}^u\lfloor\mathbf{e}_i\cdot\mathbf{x}\rfloor\geq 1.
$$ 
Hence, if $e$ and $f$ satisfy the conditions of Theorem \ref{theo gene}, then Proposition \ref{propo Lucas} implies that, for all primes $p$, $\mathfrak{S}_{e,f}$ has the $p$-Lucas property and $\mathfrak{S}_{e,f}(0)=1$ is invertible in $\mathbb{Z}_p$. Thereby, to prove Theorem \ref{theo gene}, it remains to prove that $\mathfrak{S}_{e,f}$ satisfies Condition $(a)$ in Proposition \ref{propo Ap} with
$$
\mathfrak{B}=\big\{\mathfrak{S}_{e,f}^g\,:\,g\in\mathfrak{F}_p^d\big\}.
$$
First we prove that some special functions belong to $\mathfrak{F}_p^1$.

\subsection{Special functions in $\mathfrak{F}_p^1$}

For all primes $p$, we write $|\cdot|_p$ for the ultrametric norm on $\mathbb{Q}_p$ (the field of $p$-adic numbers) defined by $|a|_p:=p^{-v_p(a)}$. Note that $(\mathbb{Z}_p,|\cdot|_p)$ is a compact space. Furthermore, if $(c_n)_{n\geq 0}$ is a $\mathbb{Z}_p$-valued sequence, then $\sum_{n=0}^{\infty}c_n$ is convergent in $(\mathbb{Z}_p,|\cdot|_p)$ if and only if $|c_n|_p$ tends to $0$ as $n$ tends to infinity. In addition, if $\sum_{n=0}^{\infty}c_n$ converges, then $(c_n)_{n\in\mathbb{N}}$ is a summable family in $(\mathbb{Z}_p,|\cdot|_p)$.
\medskip

In the rest of the article, for all primes $p$ and all positive integers $k$, we set $\Psi_{p,k,0}(0)=1$, $\Psi_{p,k,i}(0)=0$ for $i\geq 1$ and, for all natural integers $i$ and $m$, $m\geq 1$, we set 
$$
\Psi_{p,k,i}(m):=(-1)^i\sigma_{m,i}\left(\frac{1}{k},\frac{1}{k+p},\dots,\frac{1}{k+(m-1)p}\right), 
$$
where $\sigma_{m,i}$ is the $i$-th elementary symmetric polynomial of $m$ variables. Let us remind to the reader that, for all natural integers $m$ and $i$ satisfying $i>m\geq 1$, we have $\sigma_{m,i}=0$. 

The aim of this section is to prove that, for all primes $p$, all $k$ in $\{1,\dots,p-1\}$ and all natural integers $i$, we have
\begin{equation}\label{fou}
i!\Psi_{p,k,i}\in\mathfrak{F}_p^1.
\end{equation}

\begin{proof}[Proof of \eqref{fou}]
Throughout this proof, we fix a prime number $p$ and an integer $k$ in $\{1,\dots,p-1\}$. Furthermore, for all nonnegative integers $i$, we use $\Psi_i$ as a shorthand for $\Psi_{p,k,i}$ and $\mathbb{N}_{\geq i}$ as a shorthand for the set of integers larger than or equal to $i$. We shall prove \eqref{fou} by induction on $i$. To that end, for all natural integers $i$, we write $\mathcal{A}_i$ for the following assertion: 
\begin{quote}
``There exists a sequence $(T_{i,r})_{r\geq 0}$ of polynomial functions with coefficients in $\mathbb{Z}_p$ which converges uniformly to $i!\Psi_i$ on $\mathbb{N}$''.
\end{quote}

First, observe that, for all natural integers $m$, we have $\Psi_0(m)=1$, so that Assertion $\mathcal{A}_0$ is true. Let $i$ be a fixed positive integer such that assertions $\mathcal{A}_0,\dots,\mathcal{A}_{i-1}$ are true. According to the Newton-Girard formulas, for all integers $m\geq i$, we have
$$
i(-1)^i\sigma_{m,i}(X_1,\dots,X_m)=-\sum_{t=1}^i(-1)^{i-t}\sigma_{m,i-t}(X_1,\dots,X_m)\Lambda_t(X_1,\dots,X_m),
$$
where $\Lambda_t(X_1,\dots,X_m):=X_1^t+\cdots+X_m^t$. Thereby, for all integers $m\geq i$, we have
\begin{equation}\label{remind}
i\Psi_{i}(m)=-\sum_{t=1}^i\Psi_{i-t}(m)\Lambda_t\left(\frac{1}{k},\dots,\frac{1}{k+(m-1)p}\right).
\end{equation}

For all natural integers $j$ and $t$, we have
\begin{equation}\label{conv1}
\frac{1}{(k+jp)^{t}}=\frac{1}{k^{t}}\frac{1}{(1+\frac{j}{k}p)^{t}}=\frac{1}{k^{t}}+\sum_{s=1}^{\infty}\frac{(-1)^s}{k^{t}}\binom{t-1+s}{s}\left(\frac{j}{k}\right)^sp^s,
\end{equation}
where the right hand side of \eqref{conv1} is a convergent series in $(\mathbb{Z}_p,|\cdot|_p)$ because $k$ is invertible in $\mathbb{Z}_p$. Therefore, we obtain that
\begin{align}
\Lambda_t\left(\frac{1}{k},\dots,\frac{1}{k+(m-1)p}\right)
&=\frac{m}{k^t}+\sum_{j=0}^{m-1}\sum_{s=1}^{\infty}\frac{(-1)^s}{k^t}\binom{t-1+s}{s}\left(\frac{j}{k}\right)^sp^s\notag\\
&=\frac{m}{k^t}+\sum_{s=1}^{\infty}\frac{(-1)^s}{k^{t+s}}\binom{t-1+s}{s}p^s\left(\sum_{j=0}^{m-1}j^s\right).\label{ecrit1}
\end{align}
According to Faulhaber's formula, for all positive integers $s$, we have
$$
p^s\sum_{j=0}^{m-1}j^s=\sum_{c=1}^{s+1}(-1)^{s+1-c}\binom{s+1}{c}p^s\frac{B_{s+1-c}}{s+1}(m-1)^{c},
$$
where $B_k$ is the $k$-th first Bernoulli number. For all positive integers $s$ and $t$, we set $R_{0,t}(X):=X/k^t$ and
$$
R_{s,t}(X):=\frac{1}{k^{t+s}}\binom{t-1+s}{s}\sum_{c=1}^{s+1}(-1)^{1-c}\binom{s+1}{c}p^s\frac{B_{s+1-c}}{s+1}(X-1)^c,
$$
so that
$$
\Lambda_t\left(\frac{1}{k},\dots,\frac{1}{k+(m-1)p}\right)=\sum_{s=0}^\infty R_{s,t}(m).
$$
In the rest of this article, for all polynomials $P(X)=\sum_{n=0}^Na_nX^n$ in $\mathbb{Z}_p[X]$, we set
$$
\|P\|_p:=\max\big\{|a_n|_p\,:\,0\leq n\leq N\big\}.
$$
We claim that, for all natural integers $s$ and $t$, $t\geq 1$, we have
\begin{equation}\label{plus}
R_{s,t}(X)\in\mathbb{Z}_p[X],\quad\|R_{s,t}\|_p\underset{s\rightarrow\infty}{\longrightarrow}0\quad\textup{and}\quad R_{s,t}(0)=0.
\end{equation}
Indeed, on the one hand, if $p=2$ and $s=1$, then we have
$$
R_{1,t}(X)=\frac{-t}{k^{t+1}}\big(X-1+(X-1)^2\big)\in X\mathbb{Z}_2[X].
$$
On the other hand, if $p\geq 3$ or $s\geq 2$, then we have $p^s>s+1$ so that $v_p(s+1)\leq s-1$. Furthermore, according to the von Staudt-Clausen theorem, we have $v_p(B_{s+1-c})\geq -1$. Thus, the coefficients of $R_{s,t}(X)$ belong to $\mathbb{Z}_p$. To be more precise, we have $v_p(s+1)\leq\log_p(s+1)$, so that $\|R_{s,t}\|_p\underset{s\rightarrow\infty}{\longrightarrow}0$ as expected. In addition, we have
\begin{align*}
R_{s,t}(0)
&=-\frac{p^s}{(s+1)k^{t+s}}\binom{t-1+s}{s}\sum_{c=1}^{s+1}\binom{s+1}{c}B_{s+1-c}\\
&=-\frac{p^s}{(s+1)k^{t+s}}\binom{t-1+s}{s}\sum_{d=0}^{s}\binom{s+1}{d}B_d=0,
\end{align*}
where we used the well known relation satisfied by the Bernoulli numbers 
$$
\sum_{d=0}^{s}\binom{s+1}{d}B_d=0,\quad(s\geq 1).
$$
\medskip

According to $\mathcal{A}_0,\dots,\mathcal{A}_{i-1}$, for all $j$ in $\{0,\dots,i-1\}$, there exists a sequence $(T_{j,r})_{r\geq 0}$ of polynomial functions with coefficients in $\mathbb{Z}_p$ which converges uniformly to $j!\Psi_j$ on $\mathbb{N}$. 
According to \eqref{remind} and \eqref{plus}, for all natural integers $N$, there exists $S_N$ in $\mathbb{N}$ such that, for all $r\geq S_N$ and all $m\geq i$, we have
$$
i!\Psi_i(m)\equiv-\sum_{t=1}^i\frac{(i-1)!}{(i-t)!}T_{i-t,r}(m)\sum_{s=0}^rR_{s,t}(m)\mod p^N\mathbb{Z}_p.
$$
Thus, the sequence $(T_{i,r})_{r\geq 0}$ of polynomial functions with coefficients in $\mathbb{Z}_p$, defined by
\begin{equation}\label{defT}
T_{i,r}(x):=-\sum_{t=1}^i\frac{(i-1)!}{(i-t)!}T_{i-t,r}(x)\sum_{s=0}^rR_{s,t}(x),\quad(x,r\in\mathbb{N}),
\end{equation}
converges uniformly to $i!\Psi_i$ on $\mathbb{N}_{\geq i}$. To prove $\mathcal{A}_i$, it suffices to show that, for all $m$ in $\{0,\dots,i-1\}$, we have 
\begin{equation}\label{tends}
T_{i,r}(m)\underset{r\rightarrow\infty}{\longrightarrow}0.
\end{equation}

Observe that Equations \eqref{defT} and \eqref{plus} lead to $T_{i,r}(0)=0$. In particular, if $i=1$, then \eqref{tends} holds. Now we assume that $i\geq 2$. For all $m\geq 2$, we have 
\begin{align*}
\sum_{j=0}^{m}\Psi_{j}(m)X^j
&=\prod_{w=0}^{m-1}\left(1-\frac{X}{k+wp}\right)\\
&=\left(1-\frac{X}{k+(m-1)p}\right)\prod_{w=0}^{m-2}\left(1-\frac{X}{k+wp}\right)\\
&=\left(1-\frac{X}{k+(m-1)p}\right)\sum_{j=0}^{m-1}\Psi_{j}(m-1)X^j.
\end{align*}
Thereby, for all $j$ in $\{1,\dots,m-1\}$, we obtain that
$$
\Psi_{j}(m)=\Psi_{j}(m-1)-\frac{\Psi_{j-1}(m-1)}{k+(m-1)p},
$$
with
$$
\frac{1}{k+(m-1)p}=\sum_{s=0}^\infty\frac{(-1)^s}{k^{s+1}}p^s(m-1)^s.
$$
Thus, there exists a sequence $(U_r)_{r\geq 0}$ of polynomials with coefficients in $\mathbb{Z}_p$ such that, for all positive integers $N$, there exits a natural integer $S_N$ such that, for all $r\geq S_N$ and all $m\geq i+1$, we have
\begin{equation}\label{cong1}
T_{i,r}(m)\equiv T_{i,r}(m-1)-T_{i-1,r}(m-1)U_r(m-1)\mod p^N\mathbb{Z}_p.
\end{equation}

But, if $V_1(X)$ and $V_2(X)$ are polynomials with coefficients in $\mathbb{Z}_p$ and if there exists a natural integer $a$ such that, for all $m\geq a$, we have $V_1(m)\equiv V_2(m)\mod p^N\mathbb{Z}_p$, then, for all integers $n$, we have $V_1(n)\equiv V_2(n)\mod p^N\mathbb{Z}_p$. Indeed, let $n$ be an integer, there exists a natural integer $v$ such that $n+vp^N\geq a$. Thus, we obtain that 
$$
V_1(n)\equiv V_1(n+vp^N)\equiv V_2(n+vp^N)\equiv V_2(n)\mod p^N\mathbb{Z}_p. 
$$
In particular, Equation~\eqref{cong1} also holds for all positive integers $m$.

Furthermore, according to $\mathcal{A}_{i-1}$, for all $m$ in $\{0,\dots,i-2\}$, $T_{i-1,r}(m)$ tends to zero as $r$ tends to infinity. Thus, for all positive integers $N$, there exists a natural integer $S_N$ such that, for all $r\geq S_N$ and all $m$ in $\{1,\dots,i-1\}$, we have
$$
T_{i,r}(m)\equiv T_{i,r}(m-1)\mod p^N\mathbb{Z}_p.
$$
Since $T_{i,r}(0)=0$, we obtain that $T_{i,r}(m)\equiv 0\mod p^N\mathbb{Z}_p$ for all $m$ in $\{0,\dots,i-1\}$, so that \eqref{tends} holds. This finishes the induction on $i$ and proves \eqref{fou}. 
\end{proof}

\subsection{On the $p$-adic Gamma function}

For every prime $p$, we write $\Gamma_p$ for the $p$-adic Gamma function, so that, for all natural integers $n$, we have
$$
\Gamma_p(n)=(-1)^n\underset{p\nmid \lambda}{\prod_{\lambda=1}^{n-1}}\lambda.
$$
The aim of this section is to prove Proposition \ref{propo Gamma}.

\begin{proof}[Proof of Proposition \ref{propo Gamma}]
Let $p$ be a fixed prime number. For all natural integers $n$ and $m$, we have
\begin{align}
\frac{\Gamma_p\big((m+n)p\big)}{\Gamma_p(mp)\Gamma_p(np)}&=\Bigg(\underset{p\nmid \lambda}{\prod_{\lambda=np}^{(m+n)p}}\lambda\Bigg)/\Bigg(\underset{p\nmid \lambda}{\prod_{\lambda=1}^{mp}}\lambda\Bigg)\notag\\
&=\Bigg(\underset{p\nmid \lambda}{\prod_{\lambda=1}^{mp}}(np+\lambda)\Bigg)/\Bigg(\underset{p\nmid \lambda}{\prod_{\lambda=1}^{mp}}\lambda\Bigg)\notag\\
&=\underset{p\nmid \lambda}{\prod_{\lambda=1}^{mp}}\Bigg(1+\frac{np}{\lambda}\Bigg).\label{Etape Gamma}
\end{align}

Let $X,T_1,\dots,T_m$ be $m+1$ variables. Then, we have
$$
\prod_{j=1}^m(X-T_j)=X^m+\sum_{i=1}^\infty(-1)^i\sigma_{m,i}(T_1,\dots,T_m)X^{m-i}.
$$
Therefore, we obtain that
\begin{align}
\underset{p\nmid\lambda}{\prod_{\lambda=1}^{mp}}\left(1+\frac{n p}{\lambda}\right)&=\prod_{k=1}^{p-1}\prod_{\omega=0}^{m-1}\left(1+\frac{n p}{k+\omega p}\right)\notag\\
&=\prod_{k=1}^{p-1}\left(1+\sum_{i=1}^{\infty}(-1)^i\sigma_{m,i}\left(\frac{-np}{k},\cdots,\frac{-np}{k+(m-1)p}\right)\right)\notag\\
&=\prod_{k=1}^{p-1}\left(1+\sum_{i=1}^{\infty}(-1)^in^{i}p^{i}\Psi_{p,k,i}(m)\right).\label{fond5}
\end{align}
Let $k$ in $\{1,\dots,p-1\}$ be fixed. By \eqref{fou}, for all positive integers $i$, there exists a sequence $(P_{i,\ell})_{\ell\geq0}$ of polynomial functions with coefficients in $\mathbb{Z}_p$ which converges pointwise to $i!\Psi_{p,k,i}$. We fix a natural integer $K$. For all positive integers $N$, we set
$$
f_N(x,y):=1+\sum_{i=1}^{K+1}(-1)^ix^i\frac{p^i}{i!}P_{i,N}(y).
$$
If $n$ and $m$ belong to $\{0,\dots,K\}$, then we have
\begin{align*}
R_N&:=1+\sum_{i=1}^\infty (-1)^in^ip^i\Psi_{p,k,i}(m)-f_N(n,m)\\
&=\sum_{i=1}^{K+1}(-1)^in^i\frac{p^i}{i!}\big(i!\Psi_{p,k,i}(m)-P_{i,N}(m)\big)\underset{N\rightarrow\infty}{\longrightarrow}0.
\end{align*}
Furthermore, we have $f_N(x,y)\in 1+p\mathbb{Z}_p[x,y]$. Indeed, if $i=i_0+i_1p+\cdots+i_ap^a$ with $i_j$ in $\{0,\dots,p-1\}$, then we set $\mathfrak{s}_p(i):=i_0+\cdots+i_a$ so that, for all positive integers $i$, we have
$$
i-v_p(i!)=i-\frac{i-\mathfrak{s}_p(i)}{p-1}=\frac{i(p-2)+\mathfrak{s}_p(i)}{p-1}>0.
$$
Hence, by \eqref{fond5}, we obtain that there exists a function $g$ in $\mathfrak{F}_p^2$ such that, for all natural integers $n$ and $m$, we have
$$
\underset{p\nmid\lambda}{\prod_{\lambda=1}^{mp}}\left(1+\frac{n p}{\lambda}\right)=1+g(n,m)p,
$$
which, together with \eqref{Etape Gamma}, finishes the proof of Proposition \ref{propo Gamma}. 
\end{proof}

\subsection{Last step in the proof of Theorem \ref{theo gene}}

Let $\mathfrak{A}$ be the $\mathbb{Z}_p$-module spanned by $\mathfrak{S}_{e,f}$. We set $\mathfrak{B}=\{\mathfrak{S}_{e,f}^g,\:\,g\in\mathfrak{F}_p^d\}$. We shall prove that $\mathfrak{S}_{e,f}$ and $\mathfrak{B}$ satisfy Condition $(a)$ in Proposition \ref{propo Ap}. 

Obviously, $\mathfrak{B}$ is constituted of $\mathbb{Z}_p$-valued sequences and $\mathfrak{A}$ is a subset of $\mathfrak{B}$. For all $\mathbf{a}$ in $\{0,\dots,p-1\}^d$ and $\mathbf{m}$ in $\mathbb{N}^d$, we have
$$
\mathcal{Q}_{e,f}(\mathbf{a}+\mathbf{m}p)=\frac{\prod_{i=1}^u(\mathbf{e}_i\cdot\mathbf{m}p)!\prod_{k=1}^{\mathbf{e}_i\cdot\mathbf{a}}(\mathbf{e}_i\cdot\mathbf{m}p+k)}{\prod_{i=1}^v(\mathbf{f}_i\cdot\mathbf{m}p)!\prod_{k=1}^{\mathbf{f}_i\cdot\mathbf{a}}(\mathbf{f}_i\cdot\mathbf{m}p+k)}.
$$
For every natural integer $n$, we have 
$$
\frac{(np)!}{n!}=p^n(-1)^{np}\Gamma_p(np),
$$
so that we have
$$
\frac{\prod_{i=1}^u(\mathbf{e}_i\cdot\mathbf{m}p)!}{\prod_{i=1}^v(\mathbf{f}_i\cdot\mathbf{m}p)!}=p^{(|e|-|f|)\cdot\mathbf{m}}\mathcal{Q}_{e,f}(\mathbf{m})\frac{\prod_{i=1}^u(-1)^{\mathbf{e}_i\cdot\mathbf{m}p}\Gamma_p(\mathbf{e}_i\cdot\mathbf{m}p)}{\prod_{i=1}^v(-1)^{\mathbf{f}_i\cdot\mathbf{m}p}\Gamma_p(\mathbf{f}_i\cdot\mathbf{m}p)}.
$$
Furthermore, we have
\begin{multline*}
\frac{\prod_{i=1}^u\prod_{k=1}^{\mathbf{e}_i\cdot\mathbf{a}}(\mathbf{e}_i\cdot\mathbf{m}p+k)}{\prod_{i=1}^v\prod_{k=1}^{\mathbf{f}_i\cdot\mathbf{a}}(\mathbf{f}_i\cdot\mathbf{m}p+k)}\\
=\frac{\prod_{i=1}^u\prod_{k=1,p\nmid k}^{\mathbf{e}_i\cdot\mathbf{a}}(\mathbf{e}_i\cdot\mathbf{m}p+k)}{\prod_{i=1}^v\prod_{k=1,p\nmid k}^{\mathbf{f}_i\cdot\mathbf{a}}(\mathbf{f}_i\cdot\mathbf{m}p+k)}\cdot p^{\Delta_{e,f}(\mathbf{a}/p)}\frac{\prod_{i=1}^u\prod_{k=1}^{\lfloor\mathbf{e}_i\cdot\mathbf{a}/p\rfloor}(\mathbf{e}_i\cdot\mathbf{m}+k)}{\prod_{i=1}^v\prod_{k=1}^{\lfloor\mathbf{f}_i\cdot\mathbf{a}/p\rfloor}(\mathbf{f}_i\cdot\mathbf{m}+k)}.
\end{multline*}
Since $|e|=|f|$, we have
$$
\frac{\prod_{i=1}^u(-1)^{\mathbf{e}_i\cdot\mathbf{m}p}\Gamma_p(\mathbf{e}_i\cdot\mathbf{m}p)}{\prod_{i=1}^v(-1)^{\mathbf{f}_i\cdot\mathbf{m}p}\Gamma_p(\mathbf{f}_i\cdot\mathbf{m}p)}=\frac{\prod_{i=1}^u\Gamma_p(\mathbf{e}_i\cdot\mathbf{m}p)}{\prod_{i=1}^v\Gamma_p(\mathbf{f}_i\cdot\mathbf{m}p)}.
$$
Let $\alpha_1,\dots,\alpha_d$ be natural integers with $\alpha_{i_0}\geq 1$ for some $i_0$ in $\{1,\dots,d\}$. By Proposition~\ref{propo Gamma}, there exists a function $h$ in $\mathfrak{F}_p^d$ such that, for all natural integers $m_1,\dots,m_d$, we have 
$$
\frac{\Gamma_p\big((\alpha_1m_1+\cdots+\alpha_dm_d)p\big)}{\Gamma_p\big((\alpha_1m_1+\cdots+(\alpha_{i_0}-1)m_{i_0}+\cdots+\alpha_dm_d)p\big)\Gamma_p(m_{i_0}p)}=1+h(m_1,\dots,m_d)p.
$$
Hence, there exists a function $h'$ in $\mathfrak{F}_p^d$ such that, for all natural integers $m_1,\dots,m_d$, we have
$$
\frac{\Gamma_p\big((\alpha_1m_1+\cdots+\alpha_dm_d)p\big)}{\Gamma_p(m_1p)^{\alpha_1}\cdots\Gamma_p(m_dp)^{\alpha_d}}=1+h'(m_1,\dots,m_d)p.
$$
Since $f$ is only constituted by vectors $\mathbf{1}_k$, there exists $g'$ in $\mathfrak{F}_p^d$ such that, for all $\mathbf{m}$ in $\mathbb{N}^d$, we have 
$$
\frac{\prod_{i=1}^u\Gamma_p(\mathbf{e}_i\cdot\mathbf{m}p)}{\prod_{i=1}^v\Gamma_p(\mathbf{f}_i\cdot\mathbf{m}p)}=1+g'(\mathbf{m})p.
$$
Furthermore, if $k$ is an integer coprime with $p$, and $\mathbf{d}$ a vector in $\mathbb{N}^d$, then for every $\mathbf{m}$ in $\mathbb{N}^d$, we have
$$
\frac{1}{\mathbf{d}\cdot\mathbf{m}p+k}=\sum_{s=0}^\infty(-1)^s\frac{(\mathbf{d}\cdot\mathbf{m})^s}{k^{s+1}}p^s,
$$
so that there is a function $g''$ in $\mathfrak{F}_p^d$ such that, for all $\mathbf{m}$ in $\mathbb{N}^d$, we have 
$$
\frac{1}{\mathbf{d}\cdot\mathbf{m}p+k}=\frac{1}{k}+g''(\mathbf{m})p.
$$
Hence, for all $\mathbf{a}$ in $\{0,\dots,p-1\}^d$, there exist a $p$-adic integer $\lambda_{\mathbf{a}}$ and a function $g_{\mathbf{a}}$ in $\mathfrak{F}_p^d$ such that, for all $\mathbf{m}$ in $\mathbb{N}^d$, we have
$$
\frac{\prod_{i=1}^u\prod_{k=1,p\nmid k}^{\mathbf{e}_i\cdot\mathbf{a}}(\mathbf{e}_i\cdot\mathbf{m}p+k)}{\prod_{i=1}^v\prod_{k=1,p\nmid k}^{\mathbf{f}_i\cdot\mathbf{a}}(\mathbf{f}_i\cdot\mathbf{m}p+k)}=\lambda_{\mathbf{a}}+g_{\mathbf{a}}(\mathbf{m})p.
$$
Since $f$ is only constituted by vectors $\mathbf{1}_k$, for all $i$ in $\{1,\dots,v\}$, we have $\lfloor\mathbf{f}_i\cdot\mathbf{a}/p\rfloor=0$. Thereby, for all $\mathbf{a}$ in $\{0,\dots,p-1\}^d$, there exists a function $h_{\mathbf{a}}$ in $\mathbb{Z}_p+p\mathfrak{F}_p^d$, such that, for all $\mathbf{m}$ in $\mathbb{N}^d$, we have
$$
\mathcal{Q}_{e,f}(\mathbf{a}+\mathbf{m}p)=\mathcal{Q}_{e,f}(\mathbf{m})h_{\mathbf{a}}(\mathbf{m})p^{\Delta_{e,f}(\mathbf{a}/p)}\prod_{i=1}^u\prod_{k=1}^{\lfloor\mathbf{e}_i\cdot\mathbf{a}/p\rfloor}(\mathbf{e}_i\cdot\mathbf{m}+k).
$$

Furthermore, if $\lfloor\mathbf{e}_i\cdot\mathbf{a}/p\rfloor\geq 1$ for some $i$, then $\Delta_{e,f}(\mathbf{a}/p)\geq 1$. Hence we obtain that
$$
\mathbf{m}\mapsto p^{\Delta_{e,f}(\mathbf{a}/p)}\prod_{i=1}^u\prod_{k=1}^{\lfloor\mathbf{e}_i\cdot\mathbf{a}/p\rfloor}(\mathbf{e}_i\cdot\mathbf{m}+k)\in\mathbb{Z}_p+p\mathfrak{F}_p^d.
$$ 
Let $g$ be a function in $\mathfrak{F}_p^d$. For all $\mathbf{a}$ in $\{0,\dots,p-1\}^d$ and $\mathbf{m}$ in $\mathbb{N}^d$, we set 
$$
\tau_{\mathbf{a}}(\mathbf{m}):=g(\mathbf{a}+\mathbf{m}p)h_{\mathbf{a}}(\mathbf{m})p^{\Delta_{e,f}(\mathbf{a}/p)}\prod_{i=1}^u\prod_{k=1}^{\lfloor\mathbf{e}_i\cdot\mathbf{a}/p\rfloor}(\mathbf{e}_i\cdot\mathbf{m}+k), 
$$
so that $\tau_{\mathbf{a}}\in\mathbb{Z}_p+p\mathfrak{F}_p^d$. Therefore, for all $v$ in $\{0,\dots,p-1\}$ and $n$ in $\mathbb{N}$, we have
\begin{align*}
\mathfrak{S}_{e,f}^g(v+np)&=\sum_{\mathbf{0}\leq\mathbf{a}\leq\mathbf{1}(p-1)}\underset{|\mathbf{a}+\mathbf{m}p|=v+np}{\sum_{\mathbf{m}\geq\mathbf{0}}}g(\mathbf{a}+\mathbf{m}p)\mathcal{Q}_{e,f}(\mathbf{a}+\mathbf{m}p)\\
&=\sum_{\mathbf{0}\leq\mathbf{a}\leq\mathbf{1}(p-1)}\underset{|\mathbf{a}+\mathbf{m}p|=v+np}{\sum_{\mathbf{m}\geq\mathbf{0}}}\mathcal{Q}_{e,f}(\mathbf{m})\tau_{\mathbf{a}}(\mathbf{m}).
\end{align*}
If $|\mathbf{a}+\mathbf{m}p|=v+np$, then we have $|\mathbf{a}|=v+jp$ with 
$$
0\leq j\leq\min\left(n,\left\lfloor\frac{d(p-1)-v}{p}\right\rfloor\right)=:M. 
$$
Furthermore, we have $\lfloor|\mathbf{a}|/p\rfloor=j$ and there is $k$ in $\{1,\dots,d\}$ such that $\mathbf{a}^{(k)}\geq (v+jp)/d$. Since $e$ is $2$-admissible and $f$ is constituted by vectors $\mathbf{1}_k$, we obtain that 
$$
\Delta_{e,f}(\mathbf{a}/p)=\sum_{i=1}^u\left\lfloor\frac{\mathbf{e}_i\cdot\mathbf{a}}{p}\right\rfloor\geq 2j. 
$$
In particular, there is $\tau_{\mathbf{a}}'$ in $\mathfrak{F}_p^d$ such that $\tau_{\mathbf{a}}=p^{2j}\tau_{\mathbf{a}}'$. Hence, for all $\mathbf{a}$ in $\{0,\dots,p-1\}^d$, we have
$$
\mathfrak{S}_{e,f}^g(v+np)=\underset{|\mathbf{a}|=v}{\sum_{\mathbf{0}\leq\mathbf{a}\leq\mathbf{1}(p-1)}}\sum_{|\mathbf{m}|=n}\mathcal{Q}_{e,f}(\mathbf{m})\tau_{\mathbf{a}}(\mathbf{m})
+\sum_{j=1}^Mp^{2j}\underset{|\mathbf{a}|=v+jp}{\sum_{\mathbf{0}\leq\mathbf{a}\leq\mathbf{1}(p-1)}}\sum_{|\mathbf{m}|=n-j}\mathcal{Q}_{e,f}(\mathbf{m})\tau_{\mathbf{a}}'(\mathbf{m}).
$$
Therefore, there exist $A'$ in $\mathfrak{A}$ and a sequence $(B_k)_{k\geq 0}$, with $B_k$ in $\mathfrak{B}$, such that
$$
\mathfrak{S}_{e,f}^g(v+np)=A'(n)+pB_0(n)+\sum_{k=1}^\infty p^{k+1}B_k(n-k).
$$
This shows that $\mathfrak{S}_{e,f}$ and $\mathfrak{B}$ satisfy Condition $(a)$ in Proposition \ref{propo Ap}, so that Theorem \ref{theo gene} is proved.
$\hfill\square$

\address{E. Delaygue, Institut Camille Jordan, Universit\'e Claude Bernard Lyon 1, 43 boulevard du 11 Novembre 1918, 69622 Villeurbanne cedex, France. Email: delaygue@math.univ-lyon1.fr}

\end{document}